\documentclass[a4paper,10pt]{amsart}
\usepackage{amsthm,amsmath,amsfonts,amssymb}
\usepackage{aliascnt}
\usepackage{bm}
\usepackage{mathtools,mathrsfs}
\usepackage{esint}
\usepackage{todonotes}
\usepackage[pdfdisplaydoctitle,colorlinks,breaklinks,urlcolor=blue,linkcolor=blue,citecolor=blue]{hyperref}
\usepackage[nameinlink,capitalise]{cleveref}

\newcommand{\A}{\mathcal{A}}

\newcommand{\cD}{\mathcal{D}}
\newcommand{\E}{\mathcal{E}}
\newcommand{\e}{\mathrm e}
\newcommand{\F}{\mathcal{F}}
\newcommand{\I}{\mathcal{I}}
\renewcommand{\H}{\mathbb{H}}
\newcommand{\HH}{\mathcal{H}}
\newcommand{\Nsc}{\mathcal{N}}
\newcommand{\PR}{\mathcal{P}}
\newcommand{\cQ}{\mathcal{Q}}
\newcommand{\PP}{\mathbb{P}}
\newcommand{\R}{\mathbb{R}}
\newcommand{\T}{\mathbb{T}}
\newcommand{\Z}{\mathbb{Z}}
\newcommand{\Ssc}{\mathcal{S}}

\newcommand{\Psol}{\PP}

\newcommand{\Qgrad}{\mathbb Q}
\newcommand{\Pscale}{\mathscr P}
\newcommand{\Qscale}{\mathscr Q}

\newcommand{\sJ}{\mathsf J}
\newcommand{\sH}{\mathsf H}
\newcommand{\sX}{\mathsf X}
\newcommand{\sY}{\mathsf Y}

\newcommand{\sK}{\mathsf K}

\newcommand{\sA}{\mathsf A}
\newcommand{\sB}{\mathsf B}

\newcommand{\sF}{\mathsf F}
\newcommand{\sU}{\mathsf U}
\newcommand{\sV}{\mathsf V}

\newcommand{\sZ}{\mathsf Z}

\newcommand{\sM}{\mathsf M}
\newcommand{\sN}{\mathsf N}

\newcommand{\sG}{\mathsf G}
\newcommand{\sw}{\mathsf w}
\newcommand{\sg}{\mathsf g}
\newcommand{\su}{\mathsf u}

\DeclareMathOperator{\trace}{Tr}
\DeclareMathOperator{\re}{Re}
\DeclareMathOperator{\im}{Im}
\DeclareMathOperator{\id}{Id}
\DeclareMathOperator{\Ran}{Ran}
\DeclareMathOperator{\Ker}{Ker}
\renewcommand{\epsilon}{\varepsilon}
\renewcommand{\setminus}{\smallsetminus}
\newcommand{\one}{\bm{1}}

\newcommand{\set}[1]{\left\{#1\right\}}
\newcommand{\pa}[1]{\left(#1\right)}
\newcommand{\bra}[1]{\left[#1\right]}

\newcommand{\brak}[1]{\left\langle#1\right\rangle}
\newcommand{\wick}[1]{:\mathrel{#1}:}

\newtheorem{theorem}{Theorem}[section]
\newaliascnt{definition}{theorem}

\aliascntresetthe{definition}
\newaliascnt{hypothesis}{theorem}

\aliascntresetthe{hypothesis}
\newaliascnt{corollary}{theorem}
\newtheorem{corollary}[corollary]{Corollary}
\aliascntresetthe{corollary}
\newaliascnt{lemma}{theorem}
\newtheorem{lemma}[lemma]{Lemma}
\aliascntresetthe{lemma}
\newaliascnt{proposition}{theorem}
\newtheorem{proposition}[proposition]{Proposition}
\aliascntresetthe{proposition}

\theoremstyle{remark}
\newaliascnt{remark}{theorem}

\aliascntresetthe{remark}

\numberwithin{equation}{section}

\title[Exact Asymptotics 2D Euclidean Random Matching Problem]{Exact Asymptotics for the 2D Euclidean Random Matching Problem}
\author[V. Armegioiu]{Victor Armegioiu}
\address{Department of Mathematics, ETH Z\"urich, 8092 Z\"urich, Switzerland}
\email{victor.armegioiu@math.ethz.ch}
\author[M. Goldman]{Michael Goldman}
\address{CMAP, CNRS, \'Ecole polytechnique, Institut Polytechnique de Paris, 91120 Palaiseau,
France}
\email{michael.goldman@cnrs.fr}
\author[F. Grotto]{Francesco Grotto}
\address{Universit\`a di Pisa, Dipartimento di Matematica, Largo Bruno Pontecorvo 5, 56127 Pisa, Italia}
\email{francesco.grotto at unipi.it}
\author[D. Trevisan]{Dario Trevisan}
\address{Universit\`a di Pisa, Dipartimento di Matematica, Largo Bruno Pontecorvo 5, 56127 Pisa, Italia}
\email{dario.trevisan at unipi.it}

\date{\today}
\hypersetup{
  pdftitle={Exact Asymptotics for the 2D Euclidean Random Matching Problem},
  pdfauthor={Michael Goldman, Francesco Grotto, Dario Trevisan},
  pdfkeywords={p-Laplacian, white noise, random matching, optimal transport, Wiener chaos}
}

\begin{document}
\begin{abstract}
    We determine the exact first-order asymptotics of the expected optimal cost in two-dimensional random bipartite matching, for every finite power cost $q \ge 1$, on the flat torus. In the endpoint case $q=1$, this answers a question by Talagrand, in the periodic case.
    The argument involves the closely related asymptotics of the energy of the solution of the $p$-Poisson equation with a regularized white-noise source. In the limit of vanishing regularization parameter, we identify this energy as the solution of a Variational Martingale Problem on a limiting Gaussian filtration, whose value is characterized by a parabolic Monge--Amp\`ere flow.
\end{abstract}

\maketitle

%remove toc before submitting!
%\tableofcontents

%%%%%%%%%%%%%%%%%%%%%%%%%%%%%%%%%%%%%%%%%%%%%%%%%%%%%%%%%%%%%%%%%%
\section{Introduction}\label{sec:introduction}

\subsection{From Random Matching to Singular Random PDEs}

Let $(X_i)_{i\ge1}$ and $(Y_i)_{i\ge1}$ be two independent sequences of independent random points, uniformly distributed on the flat torus
$\T^2=\R^2/\Z^2$, and set
\[
    \mu_N:=\frac1N\sum_{i=1}^N\delta_{X_i},
    \qquad
    \nu_N:=\frac1N\sum_{i=1}^N\delta_{Y_i}.
\]
For $q\in[1,\infty)$, the random Euclidean bipartite matching problem asks for the asymptotic behavior of the Wasserstein distance
\[
    W_q^q(\mu_N,\nu_N)
    =\frac1N\min_{\sigma\in\mathfrak S_N}
      \sum_{i=1}^N d(X_i,Y_{\sigma(i)})^q,
\]
where $d$ is the quotient Euclidean distance on $\T^2$.  The two-dimensional case is critical.  The microscopic spacing is of order $N^{-1/2}$, but fluctuations of the empirical density persist over all scales between $N^{-1/2}$ and order one; their cumulative effect produces the logarithmic correction
\begin{equation}\label{eq:intro-matching-main}
    E [W_q^q(\mu_N,\nu_N)]\sim
    \left(\frac{\log N}{N}\right)^{q/2}.
\end{equation}
The appearance of this critical logarithm goes back to the work of
Ajtai--Koml\'os--Tusn\'ady and to subsequent developments in random matching and discrepancy theory; see \cite{AKT84,Talagrand92,Talagrand2021} and the references therein.  What is much subtler is the existence and identification of the first-order constant.  For the quadratic cost this question is by now
well understood: the Poisson-equation ansatz proposed in the statistical physics literature by \cite{CLPS14} was made rigorous in \cite{AST19}, refined on closed surfaces in \cite{AG19}, and extended to general planar domains and non-constant densities in \cite{AGT22}.  On the flat torus this yields (see also \cite{goldman2023almost} for next order asymptotics)
\begin{equation*}
     E [W_2^2(\mu_N,\nu_N)]\sim \frac{\log N}{2\pi N}.
\end{equation*}
The purpose of this paper is to determine the corresponding first-order
asymptotics for every finite power cost, including the endpoint $q=1$.

\begin{theorem}
\label{thm:matching-exact-asymptotic}
For every $q\in[1,\infty)$, there exists $\kappa_q \in (0, \infty)$ such that
\begin{equation}\label{eq:matching-exact-limit}
    \lim_{N\to\infty}
    \left(\frac{2\pi N}{\log N}\right)^{q/2}
    E \left [ W_q^q(\mu_N,\nu_N) \right ]
    =\kappa_q.
\end{equation}
\end{theorem}

For $q>1$ the limiting constant $\kappa_q$ admits an intrinsic characterization through a self-similar Monge--Amp\`ere flow, equivalently through the one-dimensional ODE in \eqref{eq:kappa-ode-intro}; for $q=1$ it is the continuous endpoint of these constants and numerical approximation gives $\kappa_1 \approx 0.8523$. In particular, we settle on $\T^2$ the existence of the renormalized expectation for every finite power cost and identify the limiting constant through a deterministic nonlinear profile. This solves Research Problem 4.5.4 in \cite{Talagrand2021}, up to replacing the torus $\T^2$ with the unit cube $[0,1]^2$.

For related results for the case $q=\infty$ we refer to \cite{otto2026quantitativehomogenizationmaximalaction}.

\subsection{From the transport problem to a nonlinear random PDE}

The starting point is the PDE viewpoint on random matching.  In the quadratic
case, after smoothing the empirical measures at a mesoscopic scale, the
optimal transport map is close to the identity and the Monge--Amp\`ere equation
can be linearized.  The density imbalance is then transported by the gradient
of a solution of Poisson's equation.  Since the smoothed empirical fluctuation
is asymptotically Gaussian, the leading matching cost is governed by the
energy of a regularized two-dimensional Gaussian free field.  This mechanism
is the core of the rigorous quadratic theory in \cite{AST19,AG19,AGT22}; see
also the quantitative Monge--Amp\`ere linearization results of
\cite{goldman2021quantitative,goldman2022fluctuation}.

For a general matching exponent $q>1$, the correct linearization is still nonlinear. The relevant potential solves
\begin{equation}\label{eq:intro-p-laplace}
    -\nabla\cdot\left(|\nabla \su|^{p-2}\nabla \su\right)=\xi,
    \quad 
    p=q/(q-1)
\end{equation}
formally with spatial white noise $\xi$ on the right-hand side.  Equivalently, the flux
\[
    \sH=-|\nabla\su|^{p-2}\nabla\su
\]
minimizes the convex $L^q$ energy among vector fields with prescribed
divergence.  The sharp transport-to-PDE comparison proved in \cite{AVT24} shows that, after mesoscopic heat
regularization, the matching problem is asymptotically equivalent to this
$q$-Poisson energy.  In particular, the remaining issue is no longer the
transport comparison itself but the asymptotic analysis of a singular random
nonlinear elliptic problem.  The same work explicitly reduces existence of the
matching limit to existence of the corresponding stochastic PDE energy limit.

The white-noise equation \eqref{eq:intro-p-laplace} has infinite energy in two
dimensions and must be regularized.  Let $\xi_\Lambda$ be white noise with
Fourier modes restricted to $0<|k|\le\Lambda$, and define the minimum flux
energy
\begin{equation}\label{eq:intro-gaussian-energy}
    \E_q(\xi_\Lambda)
    :=\inf_{\nabla\cdot H=\xi_\Lambda}
       \int_{\T^2}|H|^q.
\end{equation}
The linear Poisson flux
\[
    \sJ_\Lambda=-\nabla(-\Delta)^{-1}\xi_\Lambda
\]
has pointwise variance
\[
    m_\Lambda:=E\left [|\sJ_\Lambda(x)|^2\right]
    =\sum_{0<|k|\le\Lambda}\frac1{|k|^2}
    \sim\frac{\log\Lambda}{2\pi}.
\]
For $q=2$, $\sJ_\Lambda$ is itself the minimizer and
$E\left [\E_2(\xi_\Lambda)\right] =m_\Lambda$.  For $q\neq2$ the minimizer is a nonlinear functional of all Fourier modes, and there is no pointwise or mode-by-mode formula. The core of the paper is to prove  nevertheless that the normalized energy has a deterministic limit, see \cref{thm:qone-ultraviolet1,thm:main}:
\begin{equation}\label{eq:intro-gaussian-main}
    m_\Lambda^{-q/2}E \left [ \E_q(\xi_\Lambda) \right] 
    \longrightarrow \kappa_q,
    \qquad \Lambda\to\infty.
\end{equation}
This Gaussian ultraviolet limit is the central result from which
\eqref{eq:matching-exact-limit} is deduced.

\subsection{Scale filtration, martingale control, and Monge--Amp\`ere flow}

The main difficulty in establishing \eqref{eq:intro-gaussian-main} is that the divergence constraint is local in Fourier space while the $L^q$ minimization is nonlinear and therefore couples all frequencies.  A literal shell-by-shell optimization would lose these nonlinear interactions. To overcome this, our approach is based on an approximation by a Variational Martingale Problem in which the Fourier modes are revealed according to their logarithmic radius.  The logarithmic divergence of $m_\Lambda$ makes the accumulated shell variance the natural time variable.  After normalization, the scale variable ranges over a bounded interval, and the lattice of frequency directions becomes asymptotically uniform on the circle.

At a fixed spatial point, the coarse-to-fine filtration turns an admissible nonlinear correction into a predictable stochastic control. The solenoidal constraint has a particularly simple asymptotic form: the new shell fixes the longitudinal component in the direction $e_\theta$ of the arriving frequency, whereas the admissible correction is transverse, in the direction $e_\theta^\perp$.  Thus the limiting controlled process has the form
\begin{equation*}
    \sX_\tau
    =x
    +\int_0^\tau\int_0^\pi
      \big(e_\theta+a_s(\theta)e_\theta^\perp\big)
      \,\sw^1(ds,d\theta)
    +\int_0^\tau\int_0^\pi
      b_s(\theta)e_\theta^\perp\,\sw^2(ds,d\theta),
\end{equation*}
where $a$ and $b$ are predictable controls and $\sw^1,\sw^2$ are independent
Gaussian white noises in scale and angle.  The value function is the minimum
of $E\bra{|\sX_\tau|^q}$ over all such controls. Let us point out that the presence of $\sw^2$ and $b$ is necessary in order to represent every admissible competitor even though for the optimal control $b=0$ so the second integral drops out.

Dynamic programming then leads to an explicit deterministic equation.  If $u_q(\tau,x)$ is the value function and $H=D_x^2u_q(\tau,x)$, minimizing the quadratic variation over the transverse controls gives the angular identity
\[
    \inf_{a,b}\int_0^\pi
    \Big[
      (e_\theta+a e_\theta^\perp)^T H(e_\theta+a e_\theta^\perp)
      +b^2(e_\theta^\perp)^THe_\theta^\perp
    \Big]\frac{d\theta}{\pi}
    =\sqrt{\det H}.
\]
Consequently the Bellman equation is the parabolic Monge--Amp\`ere flow
\begin{equation}\label{eq:intro-ma-flow}
    \partial_\tau u_q=\frac12\sqrt{\det D_x^2u_q},
    \qquad
    u_q(0,x)=|x|^q.
\end{equation}
The power-law initial datum is compatible with Brownian scaling, and the
solution is self-similar:
\[
    u_q(\tau,x)
    =\tau^{q/2}\phi_q\left(\frac{|x|}{\sqrt\tau}\right).
\]
The radial profile $\phi_q$ is characterized by the ODE
\begin{equation}\label{eq:ode}
    q\phi_q(s)-s\phi_q'(s)=\sqrt{\phi_q''(s)\phi_q'(s)s^{-1}},
\end{equation}
in particular $ \kappa_q=u_q(1,0)=\phi_q(0)$. This gives a deterministic characterization of the nonlinear matching
constant and explains why $q=2$ is exceptional: only in that case is the linear Poisson flux already optimal.

Turning this formal scale picture into a theorem requires two further steps. First, we identify the scaling limit of the solenoidal constraint on finite Wiener chaos expansions and prove the corresponding lower bound for arbitrary admissible fields.  Second, we construct stationary spatial recovery sequences by lifting finite-chaos martingale controls back to the discrete Fourier noises.  These two directions show that the scale martingale problem captures the full ultraviolet asymptotics rather than merely providing a restricted class of competitors.

A central mechanism for the construction is a scale decoupling for the Leray projection as seen in  \cref{lem:largest-mode-estimate}. For a stationary lift  built from a mixture of frequencies, its Leray projection is asymptotically indistinguishable from  the lift obtained by projecting only through the finest mode present. In the limit problem this property is also apparent from \eqref{eq:homogeneous-chaos-integrand}.

Our proof bears many analogies with recent scale-by-scale homogenization approaches to the study of critical drift-diffusion equations, see \cite{armstrong2026superdiffusion,armstrong2024superdiffusive,chatzigeorgiou2025gaussian,morfe2025criticaldriftdiffusionequationintermittent}, the analogy being the strongest with \cite{morfe2025criticaldriftdiffusionequationintermittent}.
\subsection{Return to matching and the endpoint}

To pass from the Gaussian ultraviolet problem back to random matching, we
regularize the empirical measures at a mesoscopic time
$t_N=N^{-1}(\log N)^B$, $B>2$.  The result of \cite{AVT24} compares the
matching cost with the energy of the corresponding $q$-Poisson equation for $q>1$.  We
then replace the heat cutoff by a sharp Fourier cutoff and view the resulting
linear flux as a normalized sum of independent vectors in a finite-dimensional
Hilbert space.  A quantitative high-dimensional central limit theorem in
$\mathcal W_2$ due to Zhai \cite{Zhai18}, together with stability of the minimum
flux energy with respect to the linear Poisson field, replaces this empirical
forcing by cutoff white noise at an error negligible on the critical scale.
Since
\[
    \log\Lambda_N=\frac12\log N+o(\log N),
\]
the Gaussian asymptotic \eqref{eq:intro-gaussian-main} gives precisely the
factor $(\log N/(2\pi N))^{r/2}$ in
\eqref{eq:intro-matching-main}.

The endpoint $q=1$ cannot be obtained by simply inserting $q=1$ in the
strictly convex theory.  The primal problem becomes the Beckmann minimum-flow
problem in total variation and the dual problem is a pointwise Lipschitz
constraint.  We treat this endpoint directly, using a sharp truncation theorem
for gradients together with the dual formulation, and prove continuity of the
constants as $q\downarrow1$.  Heat-flow contraction yields the endpoint lower
bound for matching, while the upper bound follows by approximation from
$q>1$.  This gives the same asymptotic formula for $W_1$ and completes the
family of finite power costs.

%----------------------------------------------
\subsection{Notation and Organization of the Paper} On $\T^2:=\R^2/\Z^2$ we use the Fourier transform convention
\begin{equation*}
    f(x)=\sum_{k\in{2\pi \Z^2}}\widehat f(k)\e^{ik\cdot x},
    \qquad
    \widehat f(k)=\int_{\T^2}f(x)\e^{-ik\cdot x}\,dx.
\end{equation*}
Thus Parseval's identity reads $\int_{\T^2}|f|^2=\sum_{k\in{2\pi \Z^2}}|\widehat f(k)|^2$.
Landau's $o$ and $O$ have their usual meaning, $x\lesssim y$ is equivalent to $x=O(y)$, $x\sim y$ is equivalent to $x=y(1+o(1))$. We denote by $\Psol$ the Leray projection onto periodic solenoidal vector fields, including constants, and let $\Qgrad=\id-\Psol$: in Fourier variables,
\begin{equation}\label{eq:finite-hodge}
 \widehat{\Psol F}(k)=
 \left(\id-\frac{k\otimes k}{|k|^2}\right)\widehat F(k),
 \qquad
 \widehat{\Qgrad F}(k)=
 \frac{k\otimes k}{|k|^2}\widehat F(k)
 \quad(k\ne0),
\end{equation}
with $\widehat{\Psol F}(0)=\widehat F(0)$ and
$\widehat{\Qgrad F}(0)=0$.  Thus $\Qgrad F$ is a mean-zero gradient field and $\Psol F$ is solenoidal.
Symbols $\mathsf a,\mathsf A$ are reserved to random scalar and vector fields. We use the multiple Wiener integral convention of \cite{Janson97}.
Symbols $p,q$ are reserved to conjugate exponents $p=q/(q-1)$, and we write $A_q(z):=|z|^{q-2}z$, $A_q(0)=0$.
For a zero-mean distribution $g$ on $\T^2$ and $s\ge 1$, define
\begin{equation}\label{eq:matching-fixed-torus-energy}
    \E_s(g)
    :=\inf\left\{
        \int_{\T^2}|H|^s:
        H\in L^s(\T^2;\R^2),\ \nabla\cdot H=g
    \right\}.
\end{equation}

The paper is divided into two parts.
In \cref{sec:scale-problem}, after setting up the ultraviolet cutoff variational problem for \eqref{eq:intro-p-laplace}, we formulate the scale martingale problem and solve its Bellman equation through the Monge--Amp\`ere flow.  We then prove the Hodge limit and the stationary recovery construction, including the $q=1$ endpoint.
In \cref{sec:matching-application} we transfer the Gaussian result to the bipartite matching problem by the nonlinear PDE reduction, cutoff stability, and quantitative Gaussian replacement.  
The \cref{app:appendix} collects technical analytic and probabilistic results.

%%%%%%%%%%%%%%%%%%%%%%%%%%%%%%%%%%%%%%%%%%%%%%%%%%%%%%%%%%
\section{A Singular Random \texorpdfstring{$p$}{p}-Poisson Equation}
\label{sec:scale-problem}
Let $\Lambda\ge2\pi$, and $\xi_\Lambda$ be the zero-averaged white noise with UV cutoff at scale $\Lambda$:
\begin{equation*}
    \xi_\Lambda(x)
    =\sum_{0<|k|\le\Lambda}
      \widehat\xi(k)\e^{ik\cdot x},
    \qquad
    \widehat\xi(-k)=\overline{\widehat\xi(k)},
    \quad 
    E\bra{\widehat\xi(k)\overline{\widehat\xi(k')}}=\one_{k=k'}.
\end{equation*}
Throughout the remainder of the paper, $\Lambda$ is sufficiently large that every frequency band introduced below is nonempty. All cutoff fields are realized on a single probability space using the same family of Fourier coefficients. The UV-cutoff problem \eqref{eq:intro-gaussian-energy} is well-posed: the proof of the following is a standard argument. We  report the argument for the measurability in \cref{app:appendix}.

\begin{lemma}\label{prop:finite-cutoff-variational}
For $q\in(1,\infty)$, and every smooth $g$, the variational problem \eqref{eq:matching-fixed-torus-energy} defining $\E_q(g)$ 
has a unique minimizer $\sH$.  It is characterized by the existence of a periodic potential $\su\in W^{1,p}(\T^2)$ such that
\begin{equation}\label{eq:dual-gradient}
    |\sH|^{q-2}\sH=-\nabla\su.
\end{equation}
Equivalently,
\begin{equation}\label{eq:flux-from-potential}
    \sH=-|\nabla\su|^{p-2}\nabla\su,
\end{equation}
and $\su$ is a weak solution of
\begin{equation}\label{eq:p-laplace}
    -\nabla\cdot\pa{|\nabla\su|^{p-2}\nabla\su}=g.
\end{equation}
The potential is unique up to an additive constant.  Moreover, if 
$\sH_{q,\Lambda}$ is the potential corresponding to $g= \xi_\Lambda$, it is strongly measurable as a $L^q(\T^2;\R^2)$-valued random variable.
\end{lemma}

The linear competitor 
$\sJ_{\Lambda}=-\nabla(-\Delta)^{-1}\xi_\Lambda$,
equivalently $\widehat{\sJ}_{\Lambda}(k) =-\frac{ik}{|k|^2}\widehat\xi(k)$, is the minimizer for $q=2$. Coordinate reflections and exchange of the two axes preserve the square lattice and the radial cutoff, therefore 
\begin{equation}\label{eq:linear-variance}
    E[\sJ_{\Lambda}(x)\otimes
     \sJ_{\Lambda}(x)]
    =\frac{m_\Lambda}{2}\id,\quad 
    m_\Lambda:=E\bra{|\sJ_{\Lambda}(x)|^2}
    =\sum_{0<|k|\le\Lambda}\frac1{|k|^2}\sim \frac{\log\Lambda}{2\pi}.
\end{equation}
In particular, $\sZ_\Lambda=m_\Lambda^{-1/2}\sJ_{\Lambda}$ at every point is a centered Gaussian vector with covariance $\id/2$, and
\begin{equation*}
 E\bra{\int_{\T^2}|\sJ_{\Lambda}|^r}
 =\Gamma\left(1+\frac r2\right)m_\Lambda^{r/2},
 \qquad r>0.
\end{equation*}
The main result of this section is the following theorem.
\begin{theorem}\label{thm:main}
Fix $q\in(1,\infty)$, set $p=q/(q-1)$, and let $\sH_{q,\Lambda}$ be the
unique minimizer of \eqref{eq:matching-fixed-torus-energy} with $g=\xi_\Lambda$.  There is a constant
$\kappa_q\in(0,\infty)$, depending only on $q$, such that
\begin{equation}\label{eq:main-normalized-limit}
    \lim_{\Lambda\to\infty}m_\Lambda^{-q/2}
    E\bra{\int_{\T^2}|\sH_{q,\Lambda}|^q}
    =\kappa_q.
\end{equation}
\end{theorem}

Clearly $\kappa_2=1$. In terms of $\Lambda$, we are equivalently claiming that
\begin{equation}\label{eq:main-log-limit}
    E\bra{\int_{\T^2}|\sH_{q,\Lambda}|^q}
    \sim\frac{\kappa_q}{(2\pi)^{q/2}}(\log\Lambda)^{q/2},\qquad
    \Lambda\to\infty.
\end{equation}
The study of the ODE \eqref{eq:ode} allows the following estimates which we prove in \cref{app:appendix}:
\begin{lemma}\label{lem:kappaqestimates}
    For $q>1$, $q\neq 2$, $\Gamma\pa{1+\frac p2}^{1-q}
    \le\kappa_q
    <\Gamma\pa{1+\frac q2}$.
\end{lemma}

%-----------------------------------------
\subsection{A Variational Martingale Problem}
Define, for $\theta \in [0,\pi]$,
\begin{equation*}
    e_\theta=(\cos\theta,\sin\theta),
    \quad
    e_\theta^\perp=(-\sin\theta,\cos\theta),
    \quad
    P_\theta=e_\theta^\perp\otimes e_\theta^\perp,
    \quad
    Q_\theta=e_\theta\otimes e_\theta.
\end{equation*}
Let $(\Omega,\F,P)$ support two independent real Gaussian white noises $\sw^{1,2}$ on $[0,\infty)\times[0,\pi]$ with control measure $ds\,d\theta/\pi$, and let $(\F_s)_{s\ge0}$ be their natural completed filtration.  For $r\in[1,\infty)$, set $L_1^r:=L^r(\Omega,\F_1,P;\R^2)$. Every random vector field $\sF\in L_1^2$ has a unique representation
\begin{equation}\label{eq:white-noise-representation}
    \sF=E[\sF]+
    \sum_{a=1}^2\int_0^1\int_0^\pi
    \Phi_{\sF}^a(s,\theta)\,\sw^a(ds,d\theta),
\end{equation}
where the integrands are predictable and belong to
$L^2(\Omega\times[0,1]\times[0,\pi];\R^2)$.

\begin{lemma}\label{lem:scale-projections}
For $\sF\in L_1^2$, the linear operators
\begin{align}
 \Pscale\sF
 &=E[\sF]+
 \int_0^1\int_0^\pi
 P_\theta\Phi_{\sF}^1(s,\theta)\,\sw^1(ds,d\theta)
 +\int_0^1\int_0^\pi
 P_\theta\Phi_{\sF}^2(s,\theta)\,\sw^2(ds,d\theta),
 \label{eq:scale-solenoidal}
 \\
 \Qscale\sF
 &=
 \int_0^1\int_0^\pi
 Q_\theta\Phi_{\sF}^1(s,\theta)\,\sw^1(ds,d\theta)
 +\int_0^1\int_0^\pi
 Q_\theta\Phi_{\sF}^2(s,\theta)\,\sw^2(ds,d\theta),
\end{align}
are complementary orthogonal projections on $L_1^2$.  For every $r\in(1,\infty)$, they extend uniquely to bounded complementary projections on $L_1^r$.  Moreover,
\begin{equation}\label{eq:scale-duality}
 E[(\Pscale\sF)\cdot\sG]
 =E[\sF\cdot(\Pscale\sG)], 
 \quad \sF\in L_1^r,\, \sG\in L_1^{r'},\,
 r'=r/(r-1).
\end{equation}
Consequently,
\begin{equation}\label{eq:annihilator-scale}
 \big(\Ran_{L_1^r}\Pscale\big)^\perp
 =\Ran_{L_1^{r'}}\Qscale
 =\Ker_{L_1^{r'}}\Pscale.
\end{equation}
\end{lemma}

The proof is a standard argument we report in \cref{app:appendix}. Constants belong to the range of $\Pscale$, just as spatially constant vector fields belong to the range of the Leray projector $\Psol$. The Gaussian vector field
\begin{equation}
 \sG:=\int_0^1\int_0^\pi e_\theta\,\sw^1(ds,d\theta),
\end{equation}
satisfies $E[\sG\otimes\sG]=\frac12\id$, $\Pscale\sG=0$, and $ E[|\sG|^r]=\Gamma\left(1+\frac r2\right)$ for all $r>0$.

We now prove the infinite dimensional version of \cref{prop:finite-cutoff-variational}.
\begin{lemma}\label{lem:limiting-duality}
The infimum
\begin{equation}\label{eq:kappa-primal}
 \kappa_q
 :=\inf_{\sV\in\Ran_{L_1^q}\Pscale}
 E[|\sG+\sV|^q]
\end{equation}
is attained by a unique $\sV_q\in\sG+\Ran_{L_1^q}\Pscale$. Setting $\sU_q=\sV_q-\sG$, the supremum
\begin{equation}\label{eq:kappa-dual}
 \frac1q \kappa_q
 =\sup_{\sF\in\Ran_{L_1^p}\Qscale}
 E\left[\sG\cdot\sF-\frac1p|\sF|^p\right]
\end{equation}
is attained at $\sF_q=A_q(\sU_q)$,
and the Euler--Lagrange equation is $\Pscale A_q(\sU_q)=0$.
\end{lemma}

\begin{proof}
The affine space in \eqref{eq:kappa-primal} is closed in $L_1^q$.  Coercivity,
reflexivity, and strict convexity give existence and uniqueness.  If
$\sU_q=\sG+\sV_q$, differentiation along
$\Ran\Pscale$ gives
\begin{equation*}
 E[A_q(\sU_q)\cdot\sV]=0,
 \qquad
 \sV\in\Ran_{L_1^q}\Pscale.
\end{equation*}
By \cref{lem:scale-projections}, $\sF_q=A_q(\sU_q)$ belongs to
$\Ran_{L_1^p}\Qscale$.  
For every $\sF$ in that space, by Fenchel's inequality and \eqref{eq:annihilator-scale},
\begin{equation*}
 E\bra{\frac1q|\sU_q|^q}
 \ge E\left[\sU_q\cdot\sF-\frac1p|\sF|^p\right]
 =E\left[\sG\cdot\sF-\frac1p|\sF|^p\right].
\end{equation*}
Equality holds for $\sF=\sF_q$, proving \eqref{eq:kappa-dual}.  The Euler-Lagrange equation is the same optimality condition written using the projection.
\end{proof}

\begin{lemma}[Control representation]\label{lem:kappacontrol}
The variational problem \eqref{eq:kappa-primal} is equivalent to
\begin{equation}\label{eq:control-problem}
 \kappa_q=\inf_{c,a,b}
 E\bra{\left|
 c+\int_0^1\int_0^\pi
 \big(e_\theta+a_s(\theta)e_\theta^\perp\big)\,\sw^1(ds,d\theta)
 +\int_0^1\int_0^\pi
 b_s(\theta)e_\theta^\perp\,\sw^2(ds,d\theta)
 \right|^q},
\end{equation}
where the infimum is taken over bounded elementary predictable scalar controls $a,b$ and $c\in\R^2$.
\end{lemma}

A bounded elementary predictable control on the time interval $[0,T]$ is a finite sum of terms $\xi\,\one_{(u,v]}(s)\one_B(\theta)$ with $0\le u<v\le T$, $B\subset[0,\pi]$ a Borel set and $\xi\in L^\infty(\F_u)$. The proof requires an approximation argument. 
Given $T<\infty$ and $r\in(1,\infty)$, for a predictable pair of scalar controls $\alpha=(a,b)$ define
\begin{equation*}
 \|\alpha\|_{\H_T^r}
 :=\left(
 E\left[\left(
 \int_0^T\int_0^\pi
 \bigl(|a_s(\theta)|^2+|b_s(\theta)|^2\bigr)
 \frac{d\theta}{\pi}\,ds
 \right)^{r/2}\right]
 \right)^{1/r}.
\end{equation*}
\begin{lemma}\label{lem:predictable-density}
Pairs of bounded elementary predictable controls are dense in the space of predictable pairs having finite $\H_T^r$ norm.  If a predictable pair is bounded by $M$, its approximants may be chosen bounded by the same $M$.  If one component vanishes identically, the approximants may be chosen with that component identically zero. Moreover, if
\begin{equation}\label{eq:control-stochastic-integral}
 \I_T(a,b)
 :=\int_0^T\int_0^\pi
 a_s(\theta)e_\theta^\perp\,\sw^1(ds,d\theta)
 +\int_0^T\int_0^\pi
 b_s(\theta)e_\theta^\perp\,\sw^2(ds,d\theta),
\end{equation}
then
\begin{equation}\label{eq:bdg-forward-controls}
 \|\I_T(a,b)\|_{L^r(\Omega;\R^2)}
 \le C_r\|(a,b)\|_{\H_T^r}.
\end{equation}
Conversely, if \eqref{eq:control-stochastic-integral} is initially defined as a square-integrable stochastic integral and its terminal value belongs to $L^r$, then
\begin{equation}\label{eq:bdg-reverse-controls}
 \|(a,b)\|_{\H_T^r}
 \le C_r\|\I_T(a,b)\|_{L^r(\Omega;\R^2)}.
\end{equation}
\end{lemma}

The latter is proved in the \cref{app:appendix}.

\begin{proof}[Proof of \cref{lem:kappacontrol}]
Every terminal value displayed in \eqref{eq:control-problem}, after subtracting $\sG$, belongs to $\Ran_{L_1^q}\Pscale$, so
its infimum is not smaller than $\kappa_q$.  Conversely, let
$\sV\in\Ran_{L_1^q}\Pscale$.  Choose bounded random vectors
$\sB_n$ with $\sB_n\to\sV$ in $L^q$, and set $\sV_n=\Pscale\sB_n$.  Since $\Pscale\sV=\sV$, boundedness of
$\Pscale$ on $L^q$ gives 
\begin{equation}\label{eq:control-range-truncation}
 \sV_n\longrightarrow\sV
 \qquad\text{in }L^q.
\end{equation}
Moreover, $\sB_n\in L_1^2$, hence $\sV_n\in L_1^2$.  Its martingale representation has the form
\begin{equation}\label{eq:control-representation-approx}
 \sV_n=E[\sV_n]+
 \int_0^1\int_0^\pi
 a_{n,s}(\theta)e_\theta^\perp\,\sw^1(ds,d\theta)
 +\int_0^1\int_0^\pi
 b_{n,s}(\theta)e_\theta^\perp\,\sw^2(ds,d\theta).
\end{equation}
Indeed, applying $\Pscale$ to the martingale representation of
$\sB_n$ projects each vector integrand onto the span of
$e_\theta^\perp$.  Since the centered terminal value in
\eqref{eq:control-representation-approx} belongs to $L^q$, the reverse estimate \eqref{eq:bdg-reverse-controls} gives
\begin{equation*}
 \|(a_n,b_n)\|_{\H_1^q}
 \le C_q\|\sV_n-E[\sV_n]\|_{L^q}<\infty.
\end{equation*}
By Lemma~\ref{lem:predictable-density}, choose bounded elementary predictable pairs $(a_{n,k},b_{n,k})$ converging to $(a_n,b_n)$ in $\H_1^q$.  The forward estimate
\eqref{eq:bdg-forward-controls} then gives as $k\to \infty$,
\begin{align*}
 &\left\|
 \int_0^1\int_0^\pi
 (a_{n,k}-a_n)e_\theta^\perp\,\sw^1(ds,d\theta)
 +\int_0^1\int_0^\pi
 (b_{n,k}-b_n)e_\theta^\perp\,\sw^2(ds,d\theta)
 \right\|_{L^q}
 \\
 &\hspace{45mm}\le
 C_q\|(a_{n,k}-a_n,b_{n,k}-b_n)\|_{\H_1^q}
 \longrightarrow0.
\end{align*}
For each $n$, choose $k(n)$ so that this error is at most $n^{-1}$.
Together with \eqref{eq:control-range-truncation}, the resulting bounded elementary controls approximate $\sV$ in terminal-value $L^q$.  Hence every admissible $\sG+\sV$ is an $L^q$-limit of terminal values in \eqref{eq:control-problem}, proving the reverse inequality.
\end{proof}

%------------------------------------------------
\subsection{Monge-Ampère Flow}\label{ssec:mongeampereflow}
Let us now consider the self-similar solution of the Monge-Ampère Flow. The proofs of the results in this paragraph are elementary, and they are deferred to \cref{app:appendix}.

\begin{lemma}\label{lem:phiqexists}
    There exists a unique strictly convex $C^2$ solution of \eqref{eq:ode} with boundary conditions
    \begin{equation}\label{eq:profile-boundary}
    \phi_q'(0)=0,
    \qquad
    0<\phi_q''(0)=q\phi_q(0),
    \qquad
    \lim_{s\to\infty}s^{-q}\phi_q(s)=1.
\end{equation}
\end{lemma}

Define as above $u_q(\tau,x):=\tau^{q/2}\phi_q(\tau^{-1/2}|x|)$, $\tau>0$, $x\in\R^2$. In order to introduce the control formulation, for $M<\infty$ and $\tau>0$, let $\A_M(\tau)$ be the class of predictable pairs $(a,b)$ satisfying for $s\in [0,\tau]$
\begin{equation*}
    |a_s(\theta)|^2+|b_s(\theta)|^2\le M^2
    \quad\text{for }dP\,ds\,d\theta/\pi\text{-a.e.}
\end{equation*}
and set $\A_{\mathrm{b}}(\tau):=\bigcup_{M<\infty}\A_M(\tau)$.

\begin{lemma}\label{lem:ucontrol}
    The function $u_q$ is strictly convex in $x$, it belongs to $C^{1,2}((0,\infty)\times\R^2)$, and it solves
\begin{equation}\label{eq:parabolic-monge-ampere}
    \partial_\tau u_q
    =\frac12\sqrt{\det D_x^2u_q},
    \qquad
    \lim_{\tau\downarrow0}u_q(\tau,x)=|x|^q.
\end{equation}
For every $\tau>0$ and $x\in\R^2$,
\begin{multline}\label{eq:value-function}
    u_q(\tau,x)
    =\inf_{(a,b)\in\A_{\mathrm{b}}(\tau)}
    E\left[\left|x+\int_0^\tau\int_0^\pi
        \big(e_\theta+a_s(\theta)e_\theta^\perp\big)\,\sw^1(ds,d\theta) \right.\right.\\
    \left. \left.+\int_0^\tau\int_0^\pi
        b_s(\theta)e_\theta^\perp\,\sw^2(ds,d\theta) \right|^q\right].
\end{multline}
The same value is obtained if the infimum is restricted to bounded elementary predictable controls, or if it is taken over the terminal-value $L^q$ closure of either class.  Moreover, $\kappa_q =u_q(1,0)=\phi_q(0)$. The infimum in \eqref{eq:value-function} is approached by bounded controls with no second channel.
\end{lemma}

The Monge-Amp\`ere operator in fact arises from the following elementary matrix identity: for every positive definite symmetric matrix $H\in\R^{2\times2}$,
\begin{equation}\label{eq:angular-minimization}
    \inf_{a,b\in L^2(0,\pi)}
    \int_0^\pi
    \bra{(e_\theta+a(\theta)e_\theta^\perp)^\top  H(e_\theta+a(\theta)e_\theta^\perp)
    +b(\theta)^2(e_\theta^\perp)^\top He_\theta^\perp}
    \frac{d\theta}{\pi}
     =\sqrt{\det H}.
\end{equation}
The minimizer is explicit: $b=0$ and $a(\theta)
 =-\frac{(e_\theta^\perp)^\top He_\theta}
 {(e_\theta^\perp)^\top He_\theta^\perp}$.

%----------------------------------------------
\subsection{Stationary Lift of the Variational Martingale Problem}
We now show how the solution of the Variational Martingale Problem can be used to construct efficient competitors for the $p$-Poisson equation problem. This is the technical core of the argument, and it is based on a ``reconstruction'' of space-stationary fields using multiple Wiener integrals. Let us mention that the choice of using Wiener chaos decomposition for this construction is not essential, and there may be other approximation procedures leading to the same result.

Let us begin by defining frequency sectors
\begin{equation*}
    \Nsc^+
    =\set{n\in\Z^2:n_2>0
    \text{ or }(n_2=0,n_1>0)},\quad 
    \Nsc_\Lambda=\set{n\in\Z^2:0<|n|\le \frac{\Lambda}{2\pi}},
\end{equation*}
so that $\Nsc_\Lambda^+=\Nsc_\Lambda\cap\Nsc^+$ contains exactly one representative of pairs $\{n,-n\}$ in $\Nsc_\Lambda$. 
We then define a logarithmic scale and its associated measures on frequency sectors: for $n\in\Nsc_\Lambda^+$, let
\begin{equation*}
    s_n=\frac{\log|n|}{\log \frac{\Lambda}{2\pi}},\quad
    c_{\Lambda,n}=\frac1{\pi\sqrt{2m_\Lambda}|n|}, \quad \text{and} \quad
    \mu_\Lambda:=\sum_{n\in\Nsc_\Lambda^+}c_{\Lambda,n}^2
    \delta_{(s_n,\theta_n)}.
\end{equation*}
By construction, $\mu_\Lambda$ is a probability measure and $\max_{n\in\Nsc_\Lambda^+}c_{\Lambda,n}^2=O(1/\log \Lambda)$.
Under this measure, as $\Lambda\to \infty$, lattice points become uniformly distributed: the following result can be derived from a general principle established in \cite{widmer2012}. We report a proof in \cref{app:appendix} for completeness.

\begin{lemma}\label{lem:log-polar-equidistribution}
As $\Lambda\to\infty$, $\mu_\Lambda$ weakly converges to $\frac1\pi dsd\theta$ on $[0,1]\times [0,\pi]$. In particular, for each fixed $N\ge2$ and $\delta>0$, if
\begin{equation*}
    D_{N,\delta}
    :=\set{(s_1,\ldots,s_N):\, s_{(1)}\le\cdots\le s_{(N)},\, \text{and } 
    s_{(N)}-s_{(N-1)}\le\delta},
\end{equation*}
then
\begin{equation}\label{eq:tie-bound}
    \limsup_{\Lambda\to\infty}
    \mu_\Lambda^{\otimes N}(D_{N,\delta}\times [0,\pi]^N)
    \le N(N-1)\delta.
\end{equation}
\end{lemma}

Consider first the (normalized) linear optimal competitor for $q=2$:
its real Fourier series representation can be written as
\begin{gather*}
    \sZ_\Lambda(x)=\sum_{n\in\Nsc_\Lambda^+}
    c_{\Lambda,n}\frac{n}{|n|}\sg_n^1(x),\\
    \sg_n^1(x)=\mathsf x_n\cos(2\pi n\cdot x)
                     +\mathsf y_n\sin(2\pi n\cdot x),\\
    \sg_n^2(x)=\mathsf y_n\cos(2\pi n\cdot x)
                     -\mathsf x_n\sin(2\pi n\cdot x),\\
    \mathsf x_n=\sqrt2\,\im\widehat\xi(2\pi n),
    \qquad
    \mathsf y_n=\sqrt2\,\re\widehat\xi(2\pi n).
\end{gather*}
Competitors for $q\neq 2$ will be constructed generalizing this expansion, by considering stochastic integrals whose kernel also depend on the frequency mode direction:
\begin{equation}\label{eq:discrete-noise-field}
    \sw_{\Lambda,x}^a(f)
    =\sum_{n\in\Nsc_\Lambda^+}
    c_{\Lambda,n}f(s_n,\theta_n)\sg_n^a(x),
    \quad f\in C([0,1]\times [0,\pi]),\quad a=1,2.
\end{equation}
For fixed $x$, $\sw_{\Lambda,x}^{1,2}(f)$ are independent isonormal Gaussian processes on $L^2(\mu_\Lambda)$. 
Let us introduce the Gaussian Hilbert spaces associated to $\sw^{1,2}$ and their discrete counterparts we just defined. 
On one hand, define
\begin{equation*}
    \HH=L^2\left(
    [0,1]\times[0,\pi]\times\{1,2\},
    ds\,\frac{d\theta}{\pi}\otimes\#\right),
\end{equation*}
where $\#$ is counting measure, and, for the discrete noise, let
\begin{equation}\label{eq:discrete-gaussian-space}
    \HH_\Lambda=L^2(A_\Lambda,\nu_\Lambda),
    \quad 
    A_\Lambda=\Nsc_\Lambda^+\times\{1,2\},
    \quad
    \nu_\Lambda(\{(n,a)\})=c_{\Lambda,n}^2.
\end{equation}
If $f_d\in\R^2\otimes\HH^{\odot d}$ is represented by a continuous
symmetric kernel, $I_d(f_d)$ denotes its componentwise multiple integral with respect to $(\sw^1,\sw^2)$.  The notation $I_{d,\Lambda,x}(f_d)$ denotes the multiple integral with respect to the discrete isonormal process \eqref{eq:discrete-noise-field}, using the restriction $f_{d,\Lambda}\in\R^2\otimes\HH_\Lambda^{\odot d}$ of the same kernel.  Our normalization is
\begin{align*}
    E\bra{I_d(f_d)\cdot I_d(g_d)}
    &=\one_{d=e} d!\,\brak{f_d,g_d},
    \\
    E\bra{I_{d,\Lambda,x}(f_d)\cdot I_{d,\Lambda,x}(g_d)}
    &=\one_{d=e} d!\,\brak{f_{d,\Lambda},g_{d,\Lambda}}.
\end{align*}
Because $\nu_\Lambda$ is atomic, tuples containing repeated atoms contribute to the discrete integral, the continuous representative fixes the values of $f_{d,\Lambda}$ on diagonals.

For every fixed $d\ge1$ and $r\in(1,\infty)$, Wiener--It\^o isometry and Gaussian hypercontractivity give, uniformly in $\Lambda$ and $x$,
\begin{equation*}
    \|I_{d,\Lambda,x}(f_d)\|_{L^r}
    \le C_{d,r}\|I_{d,\Lambda,x}(f_d)\|_{L^2}
    \le C_{d,r}\|f_d\|_\infty.
\end{equation*}
For $r<2$, the first inequality follows by monotonicity of
$L^r$ norms.  Continuous functions are dense in the one-particle space, algebraic finite-rank tensors are dense in every symmetric tensor power, and finite Wiener chaoses are dense in Gaussian $L^r$. Consequently, continuous finite-rank kernels form a dense subspace in every Gaussian $L^r$, $1<r<\infty$.

Let us define the stationary extension of a finite-chaos vector by

\begin{equation}\label{eq:stationary-lift}
    \sF=c+\sum_{d=1}^D I_d(f_d),\qquad
    \sF_\Lambda(x)=c+\sum_{d=1}^D I_{d,\Lambda,x}(f_d).
\end{equation}
For a tuple $z_j=(s_j,\theta_j,a_j)$, let $M=\set{j:s_j=\max_i s_i}$ and define symmetric kernels
\begin{equation}\label{eq:max-time-kernel-projection}
    (\PR_df_d)(z_1,\ldots,z_d)
    =\frac1{|M|}\sum_{j\in M}P_{\theta_j}
    f_d(z_1,\ldots,z_d),
    \qquad
    \cQ_d=\id-\PR_d.
\end{equation}
Away from the time-tie set (which has zero continuum measure), $M=\{j_*\}$ and $\PR_d$ projects according to the unique largest time.

\begin{lemma}
\label{lem:chaos-scale-projection}
For $\sF$ as in \eqref{eq:stationary-lift},
\begin{equation}\label{eq:chaos-scale-projection}
    \Pscale\sF
    =c+\sum_{d=1}^D I_d(\PR_df_d),
    \qquad
    \Qscale\sF
    =\sum_{d=1}^D I_d(\cQ_df_d).
\end{equation}
\end{lemma}

\begin{proof}
It is enough to consider a single homogeneous chaos. The predictable integrand in the martingale representation of $I_d(f_d)$ is, for $a\in\{1,2\}$,
\begin{equation}\label{eq:homogeneous-chaos-integrand}
    \Phi_{I_d(f_d)}^a(s,\theta)
    =d\,I_{d-1}\left(
    f_d((s,\theta,a),\cdot)\,
    \one_{\{s_1,\ldots,s_{d-1}<s\}}\right).
\end{equation}
This follows by writing the symmetric integral over the union of the
$d!$ ordered time simplices.  Definition \eqref{eq:scale-solenoidal} applies $P_\theta$ to the vector output in \eqref{eq:homogeneous-chaos-integrand}.  After symmetrization, the resulting $d$-variable kernel is exactly $\PR_df_d$ almost everywhere: the projection is indexed by the variable carrying the unique largest time. The formula for $\Qscale$ follows by complementarity.
\end{proof}

The key step is to observe that the action of kernel projections $\PR_d$ on the stationary extension is close to the Leray projection onto solenoidal fields.

\begin{lemma}\label{lem:largest-mode-estimate}
Let $d\ge1$, let $f_d$ be a bounded continuous symmetric kernel, and write
\begin{equation*}
 \sF_{d,\Lambda}(x)=I_{d,\Lambda,x}(f_d),
 \qquad
 \sF_{d,\Lambda}^{\max}(x)=I_{d,\Lambda,x}(\PR_df_d).
\end{equation*}
With the convention $D_{1,\delta}=\varnothing$, for every $\delta>0$,
\begin{align}
 E\bra{\int_{\T^2}
 \left|\Psol\sF_{d,\Lambda}-\sF_{d,\Lambda}^{\max}\right|^2}
 \le C_d\|f_d\|_\infty^2
 \left(\pa{\frac{\Lambda}{2\pi}}^{-2\delta}
 +\mu_\Lambda^{\otimes d}(D_{d,\delta}\times\Theta^d)\right).
 \label{eq:largest-mode-estimate}
\end{align}
The analogous estimate holds with $\Qgrad$, $\cQ_d$, and the
same right-hand side.
\end{lemma}

\begin{proof}
We first express the error as a sum over signed frequency tuples, and
then separate the tuples according to the gap between their two largest
scales. Write
\begin{equation*}
 W_x(h)=\sum_{n\in\Nsc_\Lambda^+}\sum_{a=1}^2
 c_{\Lambda,n}h(n,a)\sg_n^a(x),
\end{equation*}
and denote its multiple integral by $I_d^x$. In the complexification of $\HH_\Lambda$, the vectors
\begin{equation*}
 \varepsilon_{n,a}=c_{\Lambda,n}^{-1}\one_{\{(n,a)\}},
 \qquad
 \eta_{n,\sigma}
 =\frac{\varepsilon_{n,1}-\mathrm i\sigma\varepsilon_{n,2}}{\sqrt2},
 \quad \sigma\in\{-1,1\},
\end{equation*}
give, respectively, the real part basis and a complex orthonormal
basis. The definitions of $\sg_n^1$ and $\sg_n^2$ give
\begin{equation*}
 W_x(\eta_{n,\sigma})
 =\frac{\sg_n^1(x)-\mathrm i\sigma\sg_n^2(x)}{\sqrt2}
 =\e^{2\pi\mathrm i\sigma n\cdot x}
   \frac{\mathsf x_n-\mathrm i\sigma\mathsf y_n}{\sqrt2}.
\end{equation*}
For ordered tuples $\mathbf n=(n_1,\ldots,n_d)$ and
$\boldsymbol\sigma=(\sigma_1,\ldots,\sigma_d)$, set
\begin{equation*}
 E_{\mathbf n,\boldsymbol\sigma}
 :=\bigotimes_{j=1}^d\eta_{n_j,\sigma_j},
 \qquad
 K=K(\mathbf n,\boldsymbol\sigma):=\sum_{j=1}^d\sigma_jn_j.
\end{equation*}
Multilinearity of Wick products now gives
\begin{equation}\label{eq:largest-mode-signed-frequency}
 I_d^x(E_{\mathbf n,\boldsymbol\sigma})
 =\wick{\prod_{j=1}^d W_x(\eta_{n_j,\sigma_j})}
 =\e^{2\pi\mathrm iK\cdot x}
   I_d^0(E_{\mathbf n,\boldsymbol\sigma}).
\end{equation}
Put $z_{n,a}=(s_n,\theta_n,a)$ and expand the restricted kernel as
\begin{align*}
    f_{d,\Lambda}
 &=\sum_{\mathbf n,\boldsymbol\sigma}
 b_{\mathbf n,\boldsymbol\sigma}\otimes
 E_{\mathbf n,\boldsymbol\sigma},\\
 b_{\mathbf n,\boldsymbol\sigma}
 &=\frac{\prod_{j=1}^dc_{\Lambda,n_j}}{2^{d/2}}
 \sum_{\mathbf a\in\{1,2\}^d}
 \left(\prod_{j=1}^d(\mathrm i\sigma_j)^{a_j-1}\right)
 f_d(z_{n_1,a_1},\ldots,z_{n_d,a_d}).
\end{align*}
The identity
$\frac12\sum_{\sigma=\pm1}(\mathrm i\sigma)^{a-1}
\overline{(\mathrm i\sigma)^{b-1}}=\one_{\{a=b\}}$
shows that, for each fixed $\mathbf n$,
\begin{align}
 \sum_{\boldsymbol\sigma}|b_{\mathbf n,\boldsymbol\sigma}|^2
 =\left(\prod_{j=1}^dc_{\Lambda,n_j}^2\right)
   \sum_{\mathbf a\in\{1,2\}^d}
   |f_d(z_{n_1,a_1},\ldots)|^2
 \le 2^d\|f_d\|_\infty^2
       \prod_{j=1}^dc_{\Lambda,n_j}^2.
 \label{eq:largest-mode-channel-parseval}
\end{align}

Set $P_0=\id$ and $P_K=\id-K\otimes K/|K|^2$ for $K\ne0$, and write
\begin{equation*}
 M(\mathbf n)=\set{j:s_{n_j}=\max_i s_{n_i}},
 \qquad
 \overline P_{\mathbf n}
 =\frac1{|M(\mathbf n)|}\sum_{j\in M(\mathbf n)}P_{n_j}.
\end{equation*}
By \eqref{eq:largest-mode-signed-frequency}, $\Psol$ acts on each term by
$P_K$. On the other hand, $\PR_d$ acts by $\overline P_{\mathbf n}$, so
\begin{align*}
 \Delta_\Lambda(x)
 :=\Psol\sF_{d,\Lambda}(x)-\sF_{d,\Lambda}^{\max}(x)
 =I_d^0\left(
 \sum_{\mathbf n,\boldsymbol\sigma}
 \e^{2\pi\mathrm iK\cdot x}
 \bigl((P_K-\overline P_{\mathbf n})b_{\mathbf n,\boldsymbol\sigma}\bigr)
 \otimes E_{\mathbf n,\boldsymbol\sigma}\right).
\end{align*}
The tensor inside the integral is symmetric, since its coefficients are
unchanged by simultaneous permutations of the pairs $(n_j,\sigma_j)$.
The Wiener--It\^o isometry and orthonormality of the product basis
therefore yield, for every $x$,
\begin{equation}\label{eq:exact-fock-norm}
 E\bra{|\Delta_\Lambda(x)|^2}
 =d!\sum_{\mathbf n,\boldsymbol\sigma}
 \left|(P_K-\overline P_{\mathbf n})
 b_{\mathbf n,\boldsymbol\sigma}\right|^2.
\end{equation}
For $d=1$, $K=\pm n_1$ and $P_K=\overline P_{\mathbf n}$, so the error is zero. We henceforth assume
$d\ge2$.

Put $\rho=\Lambda/(2\pi)>1$ and $\epsilon=\rho^{-\delta}$. For a mode
tuple $\mathbf n$, write $s(\mathbf n)=(s_{n_1},\ldots,s_{n_d})$.
If $s(\mathbf n)\notin D_{d,\delta}$, there is a unique largest index
$j_*$, and
\begin{equation*}
 \frac{|n_j|}{|n_{j_*}|}
 =\rho^{s_{n_j}-s_{n_{j_*}}}\le\epsilon
 \qquad(j\ne j_*).
\end{equation*}
Thus $\overline P_{\mathbf n}=P_{n_{j_*}}$ and, for every choice of signs,
\begin{equation*}
 K=\sigma_{j_*}n_{j_*}+H,
 \qquad H=\sum_{j\ne j_*}\sigma_jn_j,
 \qquad |H|\le(d-1)\epsilon|n_{j_*}|.
\end{equation*}
For nonzero vectors $u,v\in\R^2$, the two-dimensional projection formula is
\begin{equation*}
 (P_u-P_v)^2
 =\frac{\det(u,v)^2}{|u|^2|v|^2}\id,
 \qquad
 \|P_u-P_v\|_{\mathrm{op}}
 =\frac{|\det(u,v)|}{|u||v|}.
\end{equation*}
Applying this with $u=K$ and $v=n_{j_*}$, when $K\ne0$, gives
\begin{equation*}
 \|P_K-P_{n_{j_*}}\|_{\mathrm{op}}
 =\frac{|\det(K,H)|}{|K||n_{j_*}|}
 \le\frac{|H|}{|n_{j_*}|}
 \le(d-1)\epsilon.
\end{equation*}
If $K=0$, then $|H|=|n_{j_*}|$ and
$\|P_0-P_{n_{j_*}}\|_{\mathrm{op}}=1$, so the same bound holds.
On the remaining tuples, $0\le P_K,\overline P_{\mathbf n}\le\id$;
hence $-\id\le P_K-\overline P_{\mathbf n}\le\id$ and the difference
has operator norm at most one. We have proved the uniform bound
\begin{equation}\label{eq:largest-mode-multiplier-bound}
 \|P_K-\overline P_{\mathbf n}\|_{\mathrm{op}}^2
 \le(d-1)^2\epsilon^2+
 \one_{D_{d,\delta}}(s(\mathbf n)).
\end{equation}

Combining \eqref{eq:exact-fock-norm},
\eqref{eq:largest-mode-multiplier-bound}, and
\eqref{eq:largest-mode-channel-parseval}, and using that $\T^2$ has unit
volume, gives
\begin{align*}
 E\bra{\int_{\T^2}|\Delta_\Lambda|^2}
 &\le d!\sum_{\mathbf n}
 \left((d-1)^2\epsilon^2+\one_{D_{d,\delta}}(s(\mathbf n))\right)
 \sum_{\boldsymbol\sigma}|b_{\mathbf n,\boldsymbol\sigma}|^2\\
 &\le d!2^d\|f_d\|_\infty^2
 \sum_{\mathbf n}\left(\prod_{j=1}^dc_{\Lambda,n_j}^2\right)
 \left((d-1)^2\epsilon^2+\one_{D_{d,\delta}}(s(\mathbf n))\right)\\
 &=d!2^d\|f_d\|_\infty^2
 \left((d-1)^2\epsilon^2+
 \mu_\Lambda^{\otimes d}(D_{d,\delta}\times[0,\pi]^d)\right).
\end{align*}
The last equality uses $\sum_nc_{\Lambda,n}^2=1$ and the definition of
$\mu_\Lambda$. This proves \eqref{eq:largest-mode-estimate}.

Finally, since $\Qgrad=\id-\Psol$ and $\cQ_d=\id-\PR_d$,
\begin{equation*}
 \Qgrad\sF_{d,\Lambda}(x)-I_{d,\Lambda,x}(\cQ_df_d)
 =-\Delta_\Lambda(x),
\end{equation*}
so the two errors have exactly the same squared norm.
\end{proof}

\begin{theorem}\label{thm:hodge-limit}
Let $\sF$ be a finite-chaos cylindrical vector as in
\eqref{eq:stationary-lift}, and let $\sF_\Lambda$ be its stationary lift.  Then
\begin{equation}\label{eq:hodge-joint-convergence}
 \big(\sZ_\Lambda(0),\sF_\Lambda(0),\Psol\sF_\Lambda(0),\Qgrad\sF_\Lambda(0)\big)
 \ \Longrightarrow\ 
 \big(\sG,\sF,\Pscale\sF,\Qscale\sF\big),
\end{equation}
with convergence of every mixed moment.
The same statement holds at any fixed spatial point.
\end{theorem}

\begin{proof}
We first record convergence of the unprojected lifts.  Suppose that every
kernel is a finite sum of tensor products of continuous one-variable
functions.  Then the variables in \eqref{eq:stationary-lift}, at $x=0$,
are polynomials in a finite Gaussian vector of the form $\big(\sw_{\Lambda,0}^a(\varphi_j)\big)_{a,j}$.
By \cref{lem:log-polar-equidistribution}, its covariance matrix converges to that of $\big(\sw^a(\varphi_j)\big)_{a,j}$.  The corresponding Wick polynomials
therefore converge jointly, with all moments.  A continuous kernel on the
compact parameter space can be approximated uniformly by such finite-rank
kernels.  The Wiener--It\^o isometry and hypercontractivity, uniformly for
fixed chaos degree, remove this restriction.  Consequently, $\big(\sZ_\Lambda(0),\sF_\Lambda(0)\big)\Longrightarrow(\sG,\sF)$ with convergence of every mixed moment.

We next treat the maximum time projected kernels.  Although
$\PR_df_d$ is discontinuous on the time-tie set, it is continuous
elsewhere.  For fixed $\delta>0$, multiply it by a continuous
cutoff which vanishes when the two largest times differ by at most
$\delta$ and equals one when they differ by at least $2\delta$.
The preceding finite-rank argument applies to the cut-off kernel.  The
Wiener--It\^o isometry, \eqref{eq:tie-bound}, and boundedness of $f_d$
show that the removed part has $L^2$-norm at most
$C_d\|f_d\|_\infty\delta^{1/2}+o_\Lambda(1)$.  Hypercontractivity gives the
same conclusion in every finite $L^r$.  Sending first $\Lambda\to\infty$
and then $\delta\downarrow0$ yields the joint convergence
\begin{equation}\label{eq:max-kernel-convergence}
 \left(
 \sZ_\Lambda(0),\sF_\Lambda(0),
 c+\sum_{d=1}^D I_{d,\Lambda,0}(\PR_df_d),
 \sum_{d=1}^D I_{d,\Lambda,0}(\cQ_df_d)
 \right)
 \Longrightarrow
 \big(\sG,\sF,\Pscale\sF,\Qscale\sF\big)
\end{equation}
with all mixed moments; here we used
\cref{lem:chaos-scale-projection}.

Finally, \cref{lem:largest-mode-estimate} and \eqref{eq:tie-bound} imply, for every fixed homogeneous chaos,
\begin{equation*}
 E\bra{\int_{\T^2}
 \left|\Psol I_{d,\Lambda,\cdot}(f_d)
       -I_{d,\Lambda,\cdot}(\PR_df_d)\right|^2}
 \longrightarrow0.
\end{equation*}
Both fields in the difference are stationary, so the same $L^2$
convergence holds at every fixed spatial point. The difference belongs to the $d$-th chaos, and hypercontractivity upgrades the convergence to every finite $L^r(\Omega)$ at such a point, as well as to $L^r(\Omega\times\T^2)$.  The corresponding statement for $\Qgrad$ follows by subtraction.  Combining these estimates with \eqref{eq:max-kernel-convergence} proves \eqref{eq:hodge-joint-convergence} and convergence of every mixed moment. Stationarity gives the same statement at every fixed spatial point.
\end{proof}

\begin{corollary}\label{cor:recovery}
Let $r\in(1,\infty)$.
\begin{enumerate}
\item For every $\sV\in\Ran_{L_1^r}\Pscale$, there are
stationary periodic solenoidal random fields $\sV_\Lambda$, polynomial in the
Fourier coefficients of $\xi_\Lambda$, such that
\begin{align*}
 (\sZ_\Lambda(0),\sV_\Lambda(0))&\Longrightarrow(\sG,\sV),
 \\
 E\bra{|\sV_\Lambda(0)|^r}&\longrightarrow E\bra{|\sV|^r},
 &E\bra{|\sZ_\Lambda(0)+\sV_\Lambda(0)|^r}&\longrightarrow E\bra{|\sG+\sV|^r}.
\end{align*}
\item For every $\sF\in\Ran_{L_1^r}\Qscale$, there are
stationary periodic mean-zero gradient random fields $\sF_\Lambda$, polynomial
in the Fourier coefficients of $\xi_\Lambda$, such that
\begin{align}
 (\sZ_\Lambda(0),\sF_\Lambda(0))&\Longrightarrow(\sG,\sF),
 \label{eq:dual-recovery}\\
 E\bra{|\sF_\Lambda(0)|^r}&\longrightarrow E\bra{|\sF|^r},
 &E[\sZ_\Lambda(0)\cdot\sF_\Lambda(0)]&\longrightarrow E[\sG\cdot\sF].
\end{align}
\end{enumerate}
\end{corollary}

\begin{proof}
Choose finite-chaos cylindrical vectors $\sA_j$ converging to
$\sV$ in $L^r$.  Since $\Pscale$ is bounded and
$\Pscale\sV=\sV$,
\begin{equation*}
 \Pscale\sA_j\longrightarrow\sV
 \qquad\text{in }L^r.
\end{equation*}
Let $\sA_{j,\Lambda}$ be the stationary lift and set
$\sV_{j,\Lambda}=\Psol\sA_{j,\Lambda}$.  For fixed $j$,
\cref{thm:hodge-limit} gives joint convergence to
$(\sG,\Pscale\sA_j)$.  It also gives convergence of the two
$r$-power moments displayed in the statement: choose an integer
$m>r$, use the uniform $m$-moment bound supplied by fixed-chaos
hypercontractivity, and apply uniform integrability.

For every $j$, choose $\Lambda_j$ so large that, whenever $\Lambda\ge\Lambda_j$,
the bounded-Lipschitz distance between the relevant joint laws and the
absolute errors in both $r$-power moments are at most $j^{-1}$.
After making $\Lambda_j$ strictly increasing, take
$j(\Lambda)=\max\{j:\Lambda_j\le\Lambda\}$.  Since
$\Pscale\sA_j\to\sV$ in $L^r$, the diagonal fields
$\sV_\Lambda:=\sV_{j(\Lambda),\Lambda}$ have all the asserted convergence properties.
They are stationary, solenoidal, and polynomial in the Fourier coefficients
of $\xi_\Lambda$.

The proof for $\sF\in\Ran_{L_1^r}\Qscale$ is identical:
approximate $\sF$ by $\Qscale\sA_j$, lift $\sA_j$, and
apply $\Qgrad$.  The resulting fields are mean-zero gradients by
\eqref{eq:finite-hodge}.  For fixed $j$, the pairing converges by the
joint moment convergence in \cref{thm:hodge-limit}; as $j\to\infty$,
\begin{equation*}
 E[\sG\cdot(\Qscale\sA_j-\sF)]\longrightarrow0
\end{equation*}
by H\"older's inequality, because $\sG$ has moments of every order.
Including this pairing among the errors controlled in the diagonal choice proves \eqref{eq:dual-recovery} and the remaining assertions.
\end{proof}

\begin{proof}[Proof of \cref{thm:main}]
Set
\begin{equation*}
    e_{q,\Lambda}=m_\Lambda^{-q/2}E\bra{\int_{\T^2}|\sH_{q,\Lambda}|^q}.
\end{equation*}
We prove matching upper and lower bounds using the exact primal and dual optimizers from \cref{lem:limiting-duality}. Write
\begin{equation*}
    \sV_q:=\sU_q-\sG\in\Ran_{L_1^q}\Pscale.
\end{equation*}
Let $\sV_\Lambda$ be the solenoidal recovery sequence associated with
$\sV_q$ by \cref{cor:recovery}.  Since $\nabla\cdot\sJ_{\Lambda}=\xi_\Lambda$, the field $\sJ_{\Lambda}+\sqrt{m_\Lambda}\,\sV_\Lambda$ is admissible in \eqref{eq:matching-fixed-torus-energy}. By stationarity and primal recovery,
\begin{align}
    \limsup_{\Lambda\to\infty}e_{q,\Lambda}
    \le\lim_{\Lambda\to\infty}
    E\bra{\int_{\T^2}|\sZ_\Lambda+\sV_\Lambda|^q}
    =E\bra{|\sG+\sV_q|^q}
    =E\bra{|\sU_q|^q}
    =\kappa_q.
    \label{eq:main-limsup}
\end{align}

For the lower bound, use the exact optimizer $\sF_q=A_q(\sU_q)\in\Ran_{L_1^p}\Qscale$. By \eqref{eq:kappa-dual},
\begin{equation}\label{eq:exact-dual-value}
    E\bra{\sG\cdot\sF_q-\frac1p|\sF_q|^p}
    =\frac{\kappa_q}{q}.
\end{equation}
Let $\sF_\Lambda$ be the gradient recovery sequence associated with
$\sF_q$, and put $\sU_{q,\Lambda}:=\frac{\sH_{q,\Lambda}}{\sqrt{m_\Lambda}}$. The field $\sU_{q,\Lambda}-\sZ_\Lambda$ is solenoidal, whereas $\sF_\Lambda$ is a mean-zero gradient.  Therefore, sample by sample,
\begin{equation}\label{eq:finite-primal-dual-orthogonality}
    \int_{\T^2}\sU_{q,\Lambda}\cdot\sF_\Lambda
    =\int_{\T^2}\sZ_\Lambda\cdot\sF_\Lambda.
\end{equation}
Fenchel's inequality and \eqref{eq:finite-primal-dual-orthogonality} give
\begin{equation*}
    \frac{e_{q,\Lambda}}q
    \ge E\bra{\int_{\T^2}
    \pa{\sZ_\Lambda\cdot\sF_\Lambda-\frac1p|\sF_\Lambda|^p}}.
\end{equation*}
Stationarity and dual recovery, followed by
\eqref{eq:exact-dual-value}, yield
\begin{equation}\label{eq:main-liminf}
    \liminf_{\Lambda\to\infty}\frac{e_{q,\Lambda}}q
    \ge E\bra{\sG\cdot\sF_q-\frac1p|\sF_q|^p}
    =\frac{\kappa_q}{q}.
\end{equation}
Combining \eqref{eq:main-limsup} and \eqref{eq:main-liminf} proves
\eqref{eq:main-normalized-limit}. 
\end{proof}

%-----------------------------------------------------
\subsection{The endpoint \texorpdfstring{$q=1$}{q=1}}
\label{sec:q-one-endpoint}

The reflexive primal--dual argument used for $q>1$ is no longer available, but the recovery statement of \cref{cor:recovery} still applies in every fixed $L^r$ space with $r>1$.  The additional ingredient needed for the dual lower bound is a sharp truncation of gradients. We first record a slightly non standard dual formula for $\E_1$.

\begin{lemma}\label{lem:qone-finite-duality}
Let $g$ be smooth and mean zero, and set
$J=-\nabla(-\Delta)^{-1}g$.  Then
\begin{equation}
    \E_1(g)
    =\sup\set{
        \int_{\T^2}J\cdot\nabla\varphi:
        \varphi\in W^{1,\infty}(\T^2),
        |\nabla\varphi|\le1\ \text{a.e.}
      }.
      \label{eq:qone-finite-dual-general}
\end{equation}
\end{lemma}

\begin{proof}
By the classical dual representation of the Beckmann energy $\E_1(g)$, see e.g. \cite{Santam}
\begin{equation*}
    \E_1(g) = \sup\set{
        \int_{\T^2} g \varphi:
        \varphi\in W^{1,\infty}(\T^2),
        |\nabla\varphi|\le1\ \text{a.e.}
      }.
\end{equation*}
Since $\nabla \cdot J= g$, integration by parts concludes the proof of \eqref{eq:qone-finite-dual-general}.
\end{proof}

Set
\begin{equation}\label{eq:qone-normalized-value}
    e_{1,\Lambda}
    :=\frac{E\bra{\E_1(\xi_\Lambda)}}{\sqrt{m_\Lambda}}.
\end{equation}
By \cref{lem:qone-finite-duality},
\begin{equation}\label{eq:qone-finite-dual}
    \frac{\E_1(\xi_\Lambda)}{\sqrt{m_\Lambda}}
    =\sup\set{
        \int_{\T^2}\sZ_\Lambda\cdot\nabla\varphi:
        \varphi\in W^{1,\infty}(\T^2),
        |\nabla\varphi|\le1\ \text{a.e.}
      }.
\end{equation}

Define the closed subspace
\begin{equation}\label{eq:qone-scale-solenoidal-L1}
    \mathscr P_1
    :=\overline{
        \bigcup_{r>1}\Ran_{L_1^r}\Pscale
      }^{\,L^1(\Omega;\R^2)}
\end{equation}
and the endpoint variational constant
\begin{equation}\label{eq:qone-kappa-primal}
    \kappa_1
    :=\inf_{\sV\in\mathscr P_1}E\bra{|\sG+\sV|}.
\end{equation}
The union in \eqref{eq:qone-scale-solenoidal-L1} is a linear subspace,
since the underlying probability space has finite mass and the extensions
of $\Pscale$ are consistent on intersections of $L^r$ spaces.

\begin{lemma}
\label{lem:qone-limiting-duality}
One has
\begin{equation}\label{eq:qone-kappa-dual}
    \kappa_1
    =\sup\set{
        E[\sG\cdot\sF]:
        \sF\in L^\infty(\Omega;\R^2),
        \Pscale\sF=0,
        |\sF|\le1\ \text{a.s.}
      }.
\end{equation}
Here $\Pscale\sF$ is understood through the $L^2$ realization of the
projection.  The supremum is attained.
\end{lemma}

\begin{proof}
The quantity in \eqref{eq:qone-kappa-primal} is the distance from $-\sG$
to the closed subspace $\mathscr P_1$ of $L^1$.  The Hahn--Banach distance
formula gives the supremum over the unit ball of the annihilator
$\mathscr P_1^\perp\subset L^\infty$.

We claim that
\begin{equation}\label{eq:qone-annihilator-identification}
    \mathscr P_1^\perp
    =\set{
        \sF\in L^\infty(\Omega;\R^2):
        \Pscale\sF=0
      }.
\end{equation}
If $\Pscale\sF=0$ and
$\sV\in\Ran_{L_1^r}\Pscale$ for some $r>1$, then
$\sF\in L_1^{r'}$ and scale duality \eqref{eq:scale-duality} gives
$E[\sF\cdot\sV]=0$.  Hence $\sF$ annihilates $\mathscr P_1$.
Conversely, if $\sF\in\mathscr P_1^\perp$, then
$\Pscale\sF\in\Ran_{L_1^2}\Pscale\subset\mathscr P_1$, and the
$L^2$ orthogonality of $\Pscale$ gives
\begin{equation*}
    0=E[\sF\cdot\Pscale\sF]=E\bra{|\Pscale\sF|^2}.
\end{equation*}
This proves \eqref{eq:qone-annihilator-identification} and hence
\eqref{eq:qone-kappa-dual}.

The unit ball of $L^\infty$ is weak-* compact, the annihilator is weak-*
closed, and the objective is weak-* continuous because $\sG\in L^1$.
Therefore the supremum is attained.
\end{proof}

\begin{proposition}
\label{prop:qone-continuity}
The endpoint constant satisfies
\begin{equation*}
    0<\kappa_1\le E[|\sG|]=\Gamma\pa{\frac32},\qquad \lim_{q\downarrow1}\kappa_q=\kappa_1.
\end{equation*}
\end{proposition}

\begin{proof}
The upper bound follows by taking $\sV=0$.  For positivity, define
\begin{equation*}
    \sG_t:=\int_0^t\int_0^\pi
       e_\theta\,\sw^1(ds,d\theta),
    \qquad
    \tau:=\inf\set{t\ge0:|\sG_t|=1}\wedge1,
    \qquad
    \sF_0:=\sG_\tau.
\end{equation*}
Then $|\sF_0|\le1$ and
\begin{equation*}
    \sF_0
    =\int_0^1\int_0^\pi
      \one_{\{s\le\tau\}}e_\theta\,\sw^1(ds,d\theta),
\end{equation*}
so $\Pscale\sF_0=0$.  Orthogonality of the post-$\tau$ increment gives
\begin{equation*}
    E[\sG\cdot\sF_0]=E\bra{|\sF_0|^2}=E[\tau]>0.
\end{equation*}
The dual formula \eqref{eq:qone-kappa-dual} proves positivity.

Let $q>1$.  Every element of $\Ran_{L_1^q}\Pscale$ belongs to
$\mathscr P_1$, and Jensen's inequality therefore gives
\begin{equation*}
    \kappa_q
    =\inf_{\sV\in\Ran_{L_1^q}\Pscale}E\bra{|\sG+\sV|^q}
    \ge \kappa_1^q.
\end{equation*}
Thus $\liminf_{q\downarrow1}\kappa_q\ge\kappa_1$.  Conversely, fix
$\eta>0$.  By the definition of $\mathscr P_1$, there are $r>1$ and
$\sV\in\Ran_{L_1^r}\Pscale$ such that
\begin{equation*}
    E\bra{|\sG+\sV|}\le\kappa_1+\eta.
\end{equation*}
For $1<q<r$, this $\sV$ is admissible in the definition of $\kappa_q$, and
dominated convergence gives
\begin{equation*}
    \limsup_{q\downarrow1}\kappa_q
    \le\lim_{q\downarrow1}E\bra{|\sG+\sV|^q}
    =E\bra{|\sG+\sV|}
    \le\kappa_1+\eta.
\end{equation*}
Letting $\eta\downarrow0$ concludes the proof.
\end{proof}

The following random periodic form of M\"uller's sharp truncation theorem is the endpoint substitute for the finite-$p$ dual recovery used in \cref{thm:main}.

\begin{lemma}
\label{lem:qone-sharp-truncation}
Fix $r>2$.  Let
$h_n:\Omega\to W^{1,r}(\T^2)$ be strongly measurable and suppose that
\begin{align}
    E\bra{\int_{\T^2}
      \operatorname{dist}\pa{\nabla h_n,\overline B_1(0)}^r}
      &\longrightarrow0,
      \label{eq:qone-trunc-hyp-distance}
    \\
    \sup_n E\bra{\int_{\T^2}|\nabla h_n|^r}&<\infty.
      \label{eq:qone-trunc-hyp-moment}
\end{align}
Then there are strongly measurable $W^{1,r}(\T^2)$-valued random
periodic Lipschitz functions $\psi_n$ and deterministic numbers
$\varepsilon_n\downarrow0$ such that
\begin{align}
    |\nabla\psi_n|&\le1+\varepsilon_n
      \quad\text{a.e. on }\Omega\times\T^2,
      \label{eq:qone-trunc-uniform}
    \\
    (P\otimes dx)\set{\psi_n\ne h_n}&\longrightarrow0,
      \label{eq:qone-trunc-small-set}
    \\
    E\bra{\int_{\T^2}|\nabla\psi_n-\nabla h_n|^2}&\longrightarrow0.
      \label{eq:qone-trunc-L2}
\end{align}
Consequently, with
\begin{equation}\label{eq:qone-rescaled-potential}
    \widehat\psi_n:=\frac{\psi_n}{1+\varepsilon_n},
\end{equation}
one has $|\nabla\widehat\psi_n|\le1$ and
\begin{equation}\label{eq:qone-rescaled-L2}
    E\bra{\int_{\T^2}
      |\nabla\widehat\psi_n-\nabla h_n|^2}
      \longrightarrow0.
\end{equation}
\end{lemma}

\begin{proof}
We first derive a deterministic periodic consequence of  \cite[Theorem~4]{Muller99}.  For every $\eta>0$ there is
$\delta(\eta)>0$ such that every periodic $h\in W^{1,1}(\T^2)$ satisfying
\begin{equation}\label{eq:qone-deterministic-smallness}
    \int_{\T^2}
      \operatorname{dist}\pa{\nabla h,\overline B_1(0)}
      \le\delta(\eta)
\end{equation}
admits a periodic $\psi\in W^{1,\infty}(\T^2)$ with
\begin{equation}\label{eq:qone-deterministic-truncation}
    |\nabla\psi|\le1+\eta\quad\text{a.e.},
    \qquad
    |\set{\psi\ne h}|\le\eta.
\end{equation}

Suppose otherwise.  Then, for some $\eta_0>0$, there are periodic
$u_j\in W^{1,1}(\T^2)$ with
\begin{equation*}
    \int_{\T^2}
      \operatorname{dist}\pa{\nabla u_j,\overline B_1(0)}
      \longrightarrow0
\end{equation*}
for which no function satisfying
\eqref{eq:qone-deterministic-truncation} with $\eta=\eta_0$ exists.
Subtracting the mean, write
\begin{equation*}
    \nabla u_j=a_j+e_j,
    \qquad
    a_j:=\Pi_{\overline B_1(0)}(\nabla u_j),
    \qquad
    \|e_j\|_{L^1}\longrightarrow0.
\end{equation*}
The periodic Poincar\'e inequality gives a uniform $W^{1,1}$ bound.
After passing to a subsequence, Rellich compactness and weak-* compactness in
$L^\infty$ give
\begin{equation*}
    u_j\longrightarrow u_0\quad\text{in }L^1(\T^2),
    \qquad
    a_j\overset{*}{\rightharpoonup}a\quad\text{in }L^\infty.
\end{equation*}
Since $e_j\to0$ in $L^1$, one has $\nabla u_0=a$ distributionally.  The
closed convex unit ball is weak-* stable, so $|\nabla u_0|\le1$ almost
everywhere.  Apply M\"uller's theorem on the open fundamental square.  Using
its finite-volume conclusion, it gives functions $v_j$ that equal $u_0$
near the boundary and satisfy
\begin{equation*}
    \big\|\operatorname{dist}
       \pa{\nabla v_j,\overline B_1(0)}\big\|_{L^\infty}
       \longrightarrow0,
    \qquad
    |\set{u_j\ne v_j}|\longrightarrow0.
\end{equation*}
Because $u_0$ is periodic, each $v_j$ extends periodically across the
boundary.  For large $j$ this contradicts the defining property of $u_j$,
and proves the deterministic assertion.

We next make the choice in
\eqref{eq:qone-deterministic-truncation} measurable.  Since $r>2$, the
space $W^{1,r}(\T^2)$ is separable and embeds continuously into
$C(\T^2)$.  For fixed $\eta$, the set
\begin{equation*}
    \Gamma_\eta:=\set{(h,\psi)\in W^{1,r}\times W^{1,r}:
       \|\nabla\psi\|_{L^\infty}\le1+\eta,\ 
       |\set{h\ne\psi}|\le\eta}
\end{equation*}
is Borel.  The first constraint is closed in $W^{1,r}$, while the second is
Borel because
\begin{equation*}
    |\set{h\ne\psi}|
    =\lim_{m\to\infty}\int_{\T^2}
       \min\set{1,m|h-\psi|}
\end{equation*}
and $W^{1,r}\hookrightarrow C$.  On the Borel set defined by
\eqref{eq:qone-deterministic-smallness}, the sections of $\Gamma_\eta$ are nonempty.  The Jankov--von Neumann selection theorem
\cite[Theorem~18.1]{Kechris95} therefore gives a universally measurable choice.  Composing it with a strongly measurable
$W^{1,r}$-valued random variable, and using that $\F$ is complete, gives a strongly measurable $W^{1,r}$-valued selection.

Set
\begin{equation*}
    D_n(\omega):=\int_{\T^2}
       \operatorname{dist}\pa{\nabla h_n(\omega),\overline B_1(0)},
    \qquad
    b_n:=E [D_n].
\end{equation*}
By H\"older's inequality and \eqref{eq:qone-trunc-hyp-distance}, $b_n\to0$. For $k\ge1$, put $\eta_k=k^{-1}$ and choose
$\delta_k:=\delta(\eta_k)$.  Choose strictly increasing integers $N_k$ so that
\begin{equation*}
    b_n\le\frac{\delta_k}{k}
    \qquad\text{for every }n\ge N_k.
\end{equation*}
For $N_k\le n<N_{k+1}$, set $\varepsilon_n=\eta_k$.  On the event
$\set{D_n\le\delta_k}$ choose $\psi_n$ by
\eqref{eq:qone-deterministic-truncation}; on its complement take
$\psi_n=0$.  For the finitely many indices $n<N_1$, take $\psi_n=0$ and
$\varepsilon_n=1$.  Markov's inequality gives
\begin{equation*}
    P[D_n>\delta_k]\le\frac1k.
\end{equation*}
Consequently, with
$A_n:=\set{(\omega,x):\psi_n(\omega,x)\ne h_n(\omega,x)}$, one has
\begin{equation*}
    (P\otimes dx)(A_n)
    \le\eta_k+P[D_n>\delta_k]
    \le\frac2k\longrightarrow0,
\end{equation*}
and $|\nabla\psi_n|\le1+\varepsilon_n$ outside a null set.

Sobolev gradients agree almost everywhere on $A_n^c$.  Hence
\begin{align*}
    E\bra{\int_{\T^2}|\nabla\psi_n-\nabla h_n|^2}
    &\le
      2E\bra{\int_{A_n}|\nabla h_n|^2}
      +2(1+\varepsilon_n)^2(P\otimes dx)(A_n).
\end{align*}
The second term tends to zero.  H\"older's inequality on
$\Omega\times\T^2$ gives
\begin{equation*}
    E\bra{\int_{A_n}|\nabla h_n|^2}
    \le
    \left(E\bra{\int_{\T^2}|\nabla h_n|^r}\right)^{2/r}
    (P\otimes dx)(A_n)^{1-2/r},
\end{equation*}
which tends to zero by \eqref{eq:qone-trunc-hyp-moment}.  This proves
\eqref{eq:qone-trunc-L2}.  Finally,
\begin{equation*}
    |\nabla\widehat\psi_n-\nabla\psi_n|
    \le\varepsilon_n
\end{equation*}
almost everywhere, so \eqref{eq:qone-rescaled-L2} follows.
\end{proof}

\begin{theorem}\label{thm:qone-ultraviolet1}
One has
\begin{equation}\label{eq:qone-main-limit1}
    \lim_{\Lambda\to\infty}
    \frac{E\bra{\E_1(\xi_\Lambda)}}{\sqrt{m_\Lambda}}
    =\kappa_1.
\end{equation}
Consequently,
\begin{equation}\label{eq:qone-main-log1}
    E\bra{\E_1(\xi_\Lambda)}
    =\frac{\kappa_1}{\sqrt{2\pi}}\sqrt{\log\Lambda}
     +o\pa{\sqrt{\log\Lambda}}.
\end{equation}
\end{theorem}

\begin{proof}

For $q> 1$, set 

\begin{equation*}
\sU_{q,\Lambda}
    :=\frac{\sH_{q,\Lambda}}{\sqrt{m_\Lambda}}.\end{equation*}

By \cref{thm:main}

\begin{equation}\label{eq:qone-fixed-q-energy}
    E\bra{\int_{\T^2}|\sU_{q,\Lambda}|^q}
    \longrightarrow\kappa_q
    \qquad\text{as }\Lambda\to\infty
\end{equation}
for every fixed $q>1$. Furthermore, \cref{prop:qone-continuity} gives

\begin{equation}
    \label{eq:qone-kappa-continuity-used}
    \kappa_q\longrightarrow\kappa_1
    \qquad\text{as }q\downarrow1.
\end{equation}

We choose a strictly decreasing sequence $q_j \downarrow 1$ such that 

\begin{equation*}
 |\kappa_{q_j}-\kappa_1|\le\frac1j.
\end{equation*}

By \eqref{eq:qone-fixed-q-energy}, there are strictly increasing $\Lambda_j \uparrow \infty$ such that

\begin{equation} \label{eq:qone-diagonal-choice}
    \left|
        E\bra{\int_{\T^2}|\sU_{q_j,\Lambda}|^{q_j}}
        -\kappa_{q_j}
    \right|
    \le\frac1j
    \qquad\text{for every }\Lambda\ge\Lambda_j.
\end{equation}

For $\Lambda_j\le\Lambda<\Lambda_{j+1}$, set
\begin{equation*}
    q_\Lambda:=q_j,
    \qquad
    p_\Lambda:=\frac{q_\Lambda}{q_\Lambda-1},
    \qquad
    \sU_\Lambda:=\sU_{q_\Lambda,\Lambda},
\end{equation*}

and define $q_\Lambda$ arbitrarily for $\Lambda < \Lambda _1$. Then
 
\begin{equation} \label{eq:qone-diagonal-energy}
q_\Lambda \longrightarrow 1,
\qquad
p_\Lambda \longrightarrow \infty,
\qquad
e_\Lambda := E\bra{\int_{\T^2} | \sU_\Lambda|^{q_\Lambda}}
\longrightarrow \kappa_1.
\end{equation}

\textit{Upper bound.} By the homogeneity of the Beckmann
value,
\begin{equation*}
    \frac{\E_1(\xi_\Lambda)}{\sqrt{m_\Lambda}}
    =\E_1\pa{\frac{\xi_\Lambda}{\sqrt{m_\Lambda}}}.
\end{equation*}
The field $\sU_\Lambda$ satisfies
\begin{equation*}
    \nabla\cdot\sU_\Lambda
    =\frac{\xi_\Lambda}{\sqrt{m_\Lambda}},
\end{equation*}
and is therefore admissible in the $L^1$ formulation from
Lemma~\ref{lem:qone-finite-duality}. Consequently,
\begin{equation*}
    \frac{\E_1(\xi_\Lambda)}{\sqrt{m_\Lambda}}
    \le\int_{\T^2}|\sU_\Lambda|
\end{equation*}
sample by sample.  Since $P\otimes dx$ is a probability measure,
H\"older's inequality gives

\begin{align}
    e_{1,\Lambda}
    &\le E\bra{\int_{\T^2}|\sU_\Lambda|}
    \le
    \pa{
        E\bra{\int_{\T^2}|\sU_\Lambda|^{q_\Lambda}}
    }^{1/q_\Lambda}
    = e_\Lambda^{1/q_\Lambda}.
    \label{eq:qone-primal-comparison}
\end{align}

It follows from \eqref{eq:qone-diagonal-energy} that 
\begin{equation} \label{eq:qone-limsup}
\limsup_{\Lambda\to\infty} e_{1,\Lambda} \le \kappa_1.
\end{equation}

\textit{Lower bound.} By \cref{prop:finite-cutoff-variational}, 
\begin{equation}\label{eq:qone-euler-gradient}
    \sF_\Lambda
    :=A_{q_\Lambda}(\sU_\Lambda)
    =|\sU_\Lambda|^{q_\Lambda-2}\sU_\Lambda
\end{equation}
is a mean-zero periodic gradient.
This is also the field selected by equality in the conjugate-power Fenchel inequality. In particular, 

\begin{equation} \label{eq:qone-euler-identities}
\sU_\Lambda\cdot\sF_\Lambda
    =|\sU_\Lambda|^{q_\Lambda},
    \quad
    |\sF_\Lambda|^{p_\Lambda}
    =|\sU_\Lambda|^{q_\Lambda}.
\end{equation}

Since
\begin{equation*}
\nabla \cdot \sU_\Lambda 
= 
\nabla \cdot \sZ_\Lambda
\end{equation*}
the difference $\sU_\Lambda - \sZ_\Lambda$ is divergence-free. As $\sF_\Lambda$ is a periodic gradient,
\begin{equation*}
    \int_{\T^2} \sZ_\Lambda\cdot\ \sF_\Lambda
    =
    \int_{\T^2}\sU_ \Lambda\cdot \sF_\Lambda.
\end{equation*}
Combining with \eqref{eq:qone-euler-identities}, this yields 

\begin{align}
    E\bra{\int_{\T^2}\sZ_\Lambda\cdot\sF_\Lambda}
    &=e_\Lambda,
    \label{eq:qone-euler-pairing}
    \\
    E\bra{\int_{\T^2}|\sF_\Lambda|^{p_\Lambda}}
    &=e_\Lambda.
    \label{eq:qone-euler-moment}
\end{align}
Thus \eqref{eq:qone-diagonal-energy} gives
\begin{equation}\label{eq:qone-euler-limits}
    E\bra{\int_{\T^2}\sZ_\Lambda\cdot\sF_\Lambda}
    \longrightarrow\kappa_1,
    \qquad
    E\bra{\int_{\T^2}|\sF_\Lambda|^{p_\Lambda}}
    \longrightarrow\kappa_1.
\end{equation}

It then remains to replace $\sF_\Lambda$ by an exactly unit-bounded gradient. Let $\Lambda_n \to \infty$ be arbitrary and fix $r > 2$. Set

\begin{equation*}
    \sF_n := \sF_{\Lambda_n},
    \qquad
    p_n := p_{\Lambda_n}.
\end{equation*}

After discarding finitely many terms, we may assume $p_n > r$. H\"older's inequality on $\Omega \times \T^2$ and \eqref{eq:qone-euler-moment} give
 \begin{equation}\label{eq:qone-euler-fixed-moment}
    E\bra{\int_{\T^2}|\sF_n|^r}
    \le
    \pa{
        E\bra{\int_{\T^2}|\sF_n|^{p_n}}
    }^{r/p_n}
    =e_{\Lambda_n}^{r/p_n}.
\end{equation}
Since $e_{\Lambda_n}\to\kappa_1>0$ and $p_n\to\infty$, the
right-hand side tends to one.  In particular,
\begin{equation}\label{eq:qone-euler-uniform-moment}
    \sup_n E\bra{\int_{\T^2}|\sF_n|^r}<\infty.
\end{equation}

Moreover,
\begin{equation*}
    \operatorname{dist}
        \pa{\sF_n,\overline B_1(0)}
    =
    (|\sF_n|-1)_+.
\end{equation*}
For every $\varepsilon>0$,
\begin{align}
    E\bra{\int_{\T^2}
        \operatorname{dist}
        \pa{\sF_n,\overline B_1(0)}^r}
    &\le
      \varepsilon^r
      +(1+\varepsilon)^{r-p_n}
        E\bra{\int_{\T^2}|\sF_n|^{p_n}}
    \notag\\
    &=\varepsilon^r
      +(1+\varepsilon)^{r-p_n}e_{\Lambda_n}.
    \label{eq:qone-euler-distance}
\end{align}

Letting first $n\to\infty$ and then $\varepsilon\downarrow0$ gives
\begin{equation}\label{eq:qone-euler-distance-limit}
    E\bra{\int_{\T^2}
        \operatorname{dist}
        \pa{\sF_n,\overline B_1(0)}^r}
    \longrightarrow0.
\end{equation}

Let $h_n$ be the mean-zero periodic potential satisfying
\begin{equation*}
    \nabla h_n=\sF_n 
\end{equation*}
By the measurability assertion in \cref{prop:finite-cutoff-variational}
 and the continuity of $U\mapsto|U|^{q_{\Lambda_n} - 2} U$, the field $\sF_n$ is strongly measurable in $L^{p_n}$ and hence in $L^r$ since $p_n > r$. The periodic Poincar\'e inequality then shows that its mean-zero potential $h_n$ is strongly measurable in $W^{1, r}$. We may therefore apply \cref{lem:qone-sharp-truncation}.
 
Let $\widehat\psi_n$ be the rescaled potentials supplied by that
lemma, and set
\begin{equation*}
    \widehat\sF_n:=\nabla\widehat\psi_n.
\end{equation*}
Then
\begin{equation}\label{eq:qone-admissible-euler-gradient}
    |\widehat\sF_n|\le1
    \quad\text{a.e.},
    \qquad
    E\bra{\int_{\T^2}
        |\widehat\sF_n-\sF_n|^2}
    \longrightarrow0.
\end{equation}

By the definition of $m_\Lambda$ and stationarity,
\begin{equation*}
    E\bra{\int_{\T^2}|\sZ_{\Lambda_n}|^2}=1.
\end{equation*}
Consequently, Cauchy--Schwarz and
\eqref{eq:qone-admissible-euler-gradient} imply
\begin{equation}\label{eq:qone-euler-pairing-stability}
    E\bra{\left|
        \int_{\T^2}
            \sZ_{\Lambda_n}\cdot
            \pa{\widehat\sF_n-\sF_n}
    \right|}
    \longrightarrow0.
\end{equation}
The field $\widehat\sF_n$ is a periodic gradient with norm at most one, and is therefore admissible in the finite-cutoff dual problem \eqref{eq:qone-finite-dual}
. Hence  
\begin{equation*}
    e_{1,\Lambda_n}
    \ge
    E\bra{\int_{\T^2}
        \sZ_{\Lambda_n}\cdot\widehat\sF_n}.
\end{equation*}
Using \eqref{eq:qone-euler-pairing},
\eqref{eq:qone-diagonal-energy}, and
\eqref{eq:qone-euler-pairing-stability}, we obtain
\begin{equation*}
    \liminf_{n\to\infty}e_{1,\Lambda_n}
    \ge\kappa_1.
\end{equation*}
Since the sequence $\Lambda_n\to\infty$ was arbitrary,
\begin{equation}\label{eq:qone-liminf}
    \liminf_{\Lambda\to\infty}e_{1,\Lambda}
    \ge\kappa_1.
\end{equation}
Combining \eqref{eq:qone-limsup} and \eqref{eq:qone-liminf} proves
\eqref{eq:qone-main-limit1}.  Finally,
\eqref{eq:qone-main-log1} follows from
\eqref{eq:linear-variance}.
\end{proof}

%%%%%%%%%%%%%%%%%%%%%%%%%%%%%%%%%%%%%%%%%%%%%%%%%%%%%%%%%%%%%%%%%%
\section{2D Random Bipartite Matching}
\label{sec:matching-application}

 In this section we prove \cref{thm:matching-exact-asymptotic}. 
For $q>1$, the proof combines the sharp transport-to-PDE comparison of \cite{AVT24} and the ultraviolet Gaussian asymptotic from \cref{thm:main}. Since the former works with  Binomial point processes with heat-flow regularization while the latter works with white noise regularized by a sharp ultraviolet cut-off, we need two replacement results,  \cref{prop:matching-heat-sharp-energy} to change the regularization kernel and \cref{prop:matching-empirical-gaussian-energy} to pass from Binomial to Gaussian noise.  At $q=1$, the two replacement estimates remain valid. Heat-flow contraction gives the lower bound, and the limit $q\downarrow1$ gives the upper bound.

%-----------------------------------------------------------------
\subsection{PDE reduction and fluctuation normalization}

For the moment fix $q\in(1,\infty)$.  Fix $B>2$ and set
\begin{equation}\label{eq:matching-mesoscopic-scale}
    t_N:=\frac{(\log N)^B}{N},
    \qquad
    \Lambda_N:=t_N^{-1/2},
    \qquad
    \alpha_N:=\frac{N}{\Lambda_N^2}=Nt_N=(\log N)^B.
\end{equation}
Let $P_t$ denote the heat semigroup on $\T^2$, and write
\begin{equation*}
    \rho_{0,N}\,dx=P_{t_N}\mu_N,
    \qquad
    \rho_{1,N}\,dx=P_{t_N}\nu_N,
    \qquad
    f_N:=\rho_{1,N}-\rho_{0,N}.
\end{equation*}
Recall the definition \eqref{eq:matching-fixed-torus-energy} of $\E_s(g)$. Let $\varphi_N$ be the mean-zero weak solution of
\begin{equation}\label{eq:matching-rprime-poisson}
    -\nabla\cdot\left(|\nabla\varphi_N|^{p-2}\nabla\varphi_N\right)
    =f_N.
\end{equation}
By \cref{prop:finite-cutoff-variational}, the flux
\begin{equation*}
    H_N:=-|\nabla\varphi_N|^{p-2}\nabla\varphi_N
\end{equation*}
is the unique minimizer for $\E_q(f_N)$ with 
\begin{equation*}
    \E_q(f_N)=\int_{\T^2}|\nabla\varphi_N|^{p}.
\end{equation*}
Since $t_N\gg N^{-1}\log N$ and
$\log(Nt_N)=B\log\log N=o(\log N)$, \cite[Theorem 1.1]{AVT24} gives
\begin{equation}\label{eq:matching-avt-reduction}
    E \bra{W_q^q(\mu_N,\nu_N)}
    =E \bra{\E_q(f_N)}
    +o\left(\left(\frac{\log N}{N}\right)^{q/2}\right).
\end{equation}
Normalize the empirical imbalance by
\begin{equation}\label{eq:matching-fluctuation-forcing}
    \zeta_N^0
    :=\sqrt{\frac N2}\,(\nu_N-\mu_N)
    =\frac1{\sqrt{2N}}\sum_{i=1}^N
      (\delta_{Y_i}-\delta_{X_i}),
    \qquad
    \zeta_N^{\mathrm h}:=P_{t_N}\zeta_N^0.
\end{equation}
Then $f_N=\sqrt{2/N}\,\zeta_N^{\mathrm h}$ and   by  homogeneity,
\begin{equation}\label{eq:matching-energy-normalization}
    \E_q(f_N)
    =\left(\frac2N\right)^{q/2}
      \E_q(\zeta_N^{\mathrm h}).
\end{equation}
The factor $2$ in this identity is the variance contribution of the two independent empirical samples. Notice that by Central Limit Theorem, $\zeta_N^0$ converges to white noise as $N\to \infty$.

%-----------------------------------------------------------------
\subsection{Quotient stability and cutoff replacement}

For a zero-mean distribution $g$ on $\T^2$, let
\begin{equation*}
    J_g:=-\nabla(-\Delta)^{-1}g,
    \qquad \nabla\cdot J_g=g.
\end{equation*}

\begin{lemma}
\label{lem:matching-quotient-stability}
For every $s\in[1,\infty)$ and every pair of zero-mean distributions $g,h$ for which $J_g,J_h\in L^s(\T^2;\R^2)$, one has
\begin{equation}\label{eq:matching-quotient-stability}
    \left|
      \E_s(g)^{1/s}-\E_s(h)^{1/s}
    \right|
    \le \|J_g-J_h\|_{L^s(\T^2)}.
\end{equation}
\end{lemma}

\begin{proof}
Let $\Ssc_s\subset L^s(\T^2;\R^2)$ be the subspace of periodic
divergence-free fields, constants included.  It is closed because divergence is continuous from $L^s$ into distributions, including when $s=1$.  Every admissible flux with divergence $g$ is of the form $J_g+V$ with $V\in\Ssc_s$.  Consequently,
\begin{equation*}
    \E_s(g)^{1/s}
    =\operatorname{dist}_{L^s}(J_g,\Ssc_s).
\end{equation*}
The distance to a fixed closed set is one-Lipschitz, which gives
\eqref{eq:matching-quotient-stability}.
\end{proof}

Let $\Pi_{\le\lambda}$ be the Fourier projection on $\T^2$ onto the modes
$0<|k|\le\lambda$, and set
\begin{equation}\label{eq:matching-sharp-forcing}
    \zeta_N^\sharp:=\Pi_{\le\Lambda_N}\zeta_N^0.
\end{equation}
We compare the heat multiplier at time $t_N=\Lambda_N^{-2}$ with this sharp ultraviolet cutoff.

For $\lambda\ge2\pi$ and $k\in{2\pi \Z^2}\setminus\{0\}$, put
\begin{equation}\label{eq:matching-cutoff-profiles}
    b_{\mathrm h}^{\lambda}(k):=\e^{-|k|^2/\lambda^2},
    \qquad
    b_\sharp^{\lambda}(k):=\one_{0<|k|\le\lambda},
    \qquad
    a^{\lambda}:=b_{\mathrm h}^{\lambda}-b_\sharp^{\lambda},
\end{equation}
and define the point-source Poisson kernel
\begin{equation}\label{eq:matching-point-kernel}
    K^b(x)
    :=\sum_{k\in{2\pi \Z^2}\setminus\{0\}}
      -\frac{ik}{|k|^2}b(k)\e^{ik\cdot x}.
\end{equation}
In particular we have 
\[
J_{\zeta_N^{\mathrm h}}=K^{b^{\Lambda_N}_{\mathrm h}} \ast \zeta^0_N \qquad \textrm{and } \qquad J_{\zeta_N^\sharp}=K^{b^{\Lambda_N}_{\sharp}} \ast \zeta^0_N.
\]
Also write
\begin{equation}\label{eq:matching-log-variance}
    m_\lambda
    :=\sum_{0<|k|\le\lambda}
      \frac1{|k|^2}.
\end{equation}
Thus $m_\lambda$ agrees with \eqref{eq:linear-variance} when
$\lambda=\Lambda$.

\begin{lemma}
\label{lem:matching-kernel-moments}
There is a universal $C<\infty$ such that, for every $\lambda\ge2\pi$,
\begin{align}
    \|K^{b_{\mathrm h}^{\lambda}}\|_\infty
    +\|K^{b_\sharp^{\lambda}}\|_\infty
    +\|K^{a^{\lambda}}\|_\infty
    &\le C\lambda,
    \label{eq:matching-kernel-linfty}\\
    \int_{\T^2}|K^{b_\sharp^{\lambda}}|^2
    &=m_\lambda,
    \label{eq:matching-kernel-sharp-l2}\\
    \int_{\T^2}|K^{b_{\mathrm h}^{\lambda}}|^2
    &\le m_\lambda+C,
    \label{eq:matching-kernel-heat-l2}\\
    \int_{\T^2}|K^{a^{\lambda}}|^2
    &\le C.
    \label{eq:matching-kernel-difference-l2}
\end{align}
Consequently, for $s\ge2$ and any one of the three profiles $b$ above,
\begin{equation}\label{eq:matching-kernel-ls}
    \int_{\T^2}|K^b|^s
    \le C_s\lambda^{s-2}
    \int_{\T^2}|K^b|^2.
\end{equation}

At $\lambda=\Lambda_N$, for every fixed $s\in(1,\infty)$,
\begin{align}
    E\bra{\|J_{\zeta_N^{\mathrm h}}\|_{L^s}^s}
    +E\bra{\|J_{\zeta_N^\sharp}\|_{L^s}^s}
    &\le C_s m_{\Lambda_N}^{s/2},
    \label{eq:matching-empirical-flux-moments}\\
    E\bra{\|J_{\zeta_N^{\mathrm h}}
      -J_{\zeta_N^\sharp}\|_{L^s}^s}
    &\le C_s.
    \label{eq:matching-cutoff-difference-moments}
\end{align}
\end{lemma}

\begin{proof}
By definition of $K^b$ and the triangle inequality,
\begin{equation*}
    \|K^{b_\sharp^{\lambda}}\|_\infty+\|K^{b_h^{\lambda}}\|_\infty
    \le\sum_{0<|k|\le\lambda}\frac1{|k|}+\sum_{k\ne0}\frac{\e^{-|k|^2/\lambda^2}}{|k|}\le C\lambda.
\end{equation*}
By triangle inequality  this proves \eqref{eq:matching-kernel-linfty}.  Parseval gives \eqref{eq:matching-kernel-sharp-l2}.  It also gives
\begin{align*}
    \int_{\T^2}|K^{b_{\mathrm h}^{\lambda}}|^2
    =\sum_{k\ne0}\frac{\e^{-2|k|^2/\lambda^2}}{|k|^2},\qquad
    \int_{\T^2}|K^{a^{\lambda}}|^2
    =\sum_{k\ne0}\frac{
      |\e^{-|k|^2/\lambda^2}-\one_{|k|\le\lambda}|^2}{|k|^2}.
\end{align*}
For the tails we notice that a simple integral comparison argument  yields 
\begin{equation}\label{eq:tail-lattice-bound}
    \sum_{|k|>\lambda}\frac{\e^{-2|k|^2/\lambda^2}}{|k|^2}
    \le\frac2\pi\int_{1/2}^\infty\frac{\e^{-u^2/2}}u\,du
    =:C_0<\infty,
\end{equation}
which is independent of $\lambda$. Together with $\e^{-2|k|^2/\lambda^2}\le1$ for $|k|\le\lambda$, this gives \eqref{eq:matching-kernel-heat-l2}. For \eqref{eq:matching-kernel-difference-l2}, note that $a^\lambda(k)=\e^{-|k|^2/\lambda^2}$ for $|k|>\lambda$, so
\eqref{eq:tail-lattice-bound} bounds $\sum_{|k|>\lambda}|a^\lambda(k)|^2/|k|^2$ by $C_0$. For $0<|k|\le\lambda$, the elementary inequality $1-\e^{-x}\le x$
($x\ge0$) gives
\begin{equation*}
    \frac{|a^\lambda(k)|^2}{|k|^2}\le\frac{|k|^2}{\lambda^4} \qquad \textrm{thus } \qquad \sum_{|k|\le\lambda}|a^\lambda(k)|^2/|k|^2\le\sum_{|k|\le\lambda} \frac{|k|^2}{\lambda^4}\le C.
\end{equation*}
This proves \eqref{eq:matching-kernel-difference-l2}.
Equation \eqref{eq:matching-kernel-ls} follows from
$\|K\|_s^s\le\|K\|_\infty^{s-2}\|K\|_2^2$.

For one of the profiles $b\in\{b_{\mathrm h}^{\lambda},b_{\sharp}^{\lambda},a^\lambda\}$, define
\begin{equation*}
    J_N^b(x)
    :=\frac1{\sqrt{2N}}\sum_{i=1}^N
      \left(K_{\Lambda_N}^b(x-Y_i)
      -K_{\Lambda_N}^b(x-X_i)\right).
\end{equation*}
These fields are, respectively, $J_{\zeta_N^{\mathrm h}}$, $J_{\zeta_N^\sharp}$, and their difference. For fixed $x$, write
\begin{equation*}
    U_i^b(x):=\frac{K_{\Lambda_N}^b(x-Y_i)
      -K_{\Lambda_N}^b(x-X_i)}{\sqrt2}.
\end{equation*}
The vectors $U_i^b(x)$ are i.i.d. and centered.  If
$I_b:=\int_{\T^2}|K_{\Lambda_N}^b|^2$, then
\begin{equation*}
    E\bra{|U_i^b(x)|^2}=I_b,
    \qquad
    E\bra{|U_i^b(x)|^s}\le C_s
      \int_{\T^2}|K_{\Lambda_N}^b|^s.
\end{equation*}
For $s\ge2$, Rosenthal's inequality and
\eqref{eq:matching-kernel-ls} yield (recall the definition \eqref{eq:matching-mesoscopic-scale} of $\alpha_N$)
\begin{equation}\label{eq:matching-rosenthal-bound}
    E\bra{|J_N^b(x)|^s}
    \le C_s\left(
      I_b^{s/2}
      +N^{1-s/2}\Lambda_N^{s-2}I_b
    \right)
    =C_s\left(I_b^{s/2}
      +\alpha_N^{1-s/2}I_b\right).
\end{equation}
The law is stationary, so the same estimate holds after spatial integration. For the difference profile, $I_b\le C$ and
$\alpha_N^{1-s/2}\le1$.  For the heat and sharp profiles,
$I_b\le m_{\Lambda_N}+C$, and the right-hand side is bounded by
$C_s m_{\Lambda_N}^{s/2}$.  This proves
\eqref{eq:matching-empirical-flux-moments} and
\eqref{eq:matching-cutoff-difference-moments} for $s\ge2$.
The remaining exponents follow from H\"older inequality.
\end{proof}
Equipped with these estimates we can now prove the first replacement result.
\begin{proposition}
\label{prop:matching-heat-sharp-energy}
For every fixed $q\in[1,\infty)$, one has
\begin{equation}\label{eq:matching-heat-sharp-energy}
    E\bra{\E_q(\zeta_N^{\mathrm h})}
    =E\bra{\E_q(\zeta_N^\sharp)}
    +o(m_{\Lambda_N}^{q/2}).
\end{equation}
\end{proposition}

\begin{proof}
Set
\begin{equation*}
    A_N:=\E_q(\zeta_N^{\mathrm h})^{1/q},
    \qquad
    B_N:=\E_q(\zeta_N^\sharp)^{1/q}.
\end{equation*}
By \cref{lem:matching-quotient-stability},
\begin{equation*}
    |A_N-B_N|
    \le\|J_{\zeta_N^{\mathrm h}}
      -J_{\zeta_N^\sharp}\|_{L^q}.
\end{equation*}
If $q=1$, then Cauchy--Schwarz  and \eqref{eq:matching-cutoff-difference-moments} with $s=2$
give
\begin{align*}
    E\bra{|A_N-B_N|}
    &\le E\bra{\|J_{\zeta_N^{\mathrm h}}
      -J_{\zeta_N^\sharp}\|_{L^1}}\\
    &\le\left(
      E\bra{\|J_{\zeta_N^{\mathrm h}}
      -J_{\zeta_N^\sharp}\|_{L^2}^2}
      \right)^{1/2}
    =O(1)=o(m_{\Lambda_N}^{1/2}).
\end{align*}
This proves the assertion at the endpoint.
Assume now that $q>1$.
By definition \eqref{eq:matching-fixed-torus-energy} of $\E_q$,
$A_N\le\|J_{\zeta_N^{\mathrm h}}\|_{L^q}$ and
$B_N\le\|J_{\zeta_N^\sharp}\|_{L^q}$.  Hence
\begin{align*}
    E\bra{|A_N^q-B_N^q|}
    &\le q\,E\left[(A_N+B_N)^{q-1}|A_N-B_N|\right]\\
    &\le q\left(E\bra{(A_N+B_N)^q}\right)^{(q-1)/q}
      \left(E\bra{|A_N-B_N|^q}\right)^{1/q}\\
    &\lesssim_q m_{\Lambda_N}^{(q-1)/2}
    =o(m_{\Lambda_N}^{q/2}),
\end{align*}
where \cref{lem:matching-kernel-moments} was used in the last line.
\end{proof}

%-----------------------------------------------------------------
\subsection{Quantitative Gaussian replacement}
We now prove our second replacement result. The main tool is the quantitative CLT theorem from \cite{Zhai18}.
\begin{proposition}
\label{prop:matching-empirical-gaussian-energy}
For every fixed $q\in[1,\infty)$, one has
\begin{equation}\label{eq:matching-empirical-gaussian-energy}
    E\bra{\E_q(\zeta_N^\sharp)}
    =E\bra{\E_q(\xi_{\Lambda_N})}
    +o(m_{\Lambda_N}^{q/2}).
\end{equation}
\end{proposition}

\begin{proof}
We start by setting some notation. Let $\mathscr H_\lambda$ be the finite-dimensional real Hilbert space of mean-zero gradient fields on $\T^2$ whose Fourier support is contained in $0<|k|\le\lambda$, equipped with the $L^2(\T^2)$ norm.  Its real dimension
$d_\lambda$ satisfies
\begin{equation}\label{eq:matching-cutoff-dimension}
    d_\lambda\le C\lambda^2.
\end{equation}
For $z\in\T^2$, define $V_z^\lambda\in\mathscr H_\lambda$ by $V_z^\lambda(x):= K^{b_\sharp^{\lambda}} (x-z)$, that is in term of Fourier coefficients
\begin{equation}\label{eq:matching-single-particle-vector}
    \widehat V_z^\lambda(k)
    =-\frac{ik}{|k|^2}\e^{-ik\cdot z}\one_{0<|k|\le\lambda}.
\end{equation}
Parseval gives
\begin{equation}\label{eq:matching-single-particle-norm}
    \|V_z^\lambda\|_{L^2}^2=m_\lambda.
\end{equation}
Let $\xi_\lambda$ be the ultraviolet-cutoff white noise with $\Lambda$ there replaced by $\lambda$,
and let
\begin{equation*}
    \sJ_\lambda:=-\nabla(-\Delta)^{-1}\xi_\lambda
    \in\mathscr H_\lambda.
\end{equation*}
We claim that there exists a coupling such that for  $\lambda=\Lambda_N$, we have  for every
$q\in[1,\infty)$,
\begin{equation}\label{eq:matching-lr-coupling}
    \left(
      E\bra{\|J_{\zeta_N^\sharp}-\sJ_{\Lambda_N}\|_{L^q}^q}
    \right)^{1/q}
    =o(\sqrt{m_{\Lambda_N}}).
\end{equation}
With this claim at hand, we can conclude the proof of \eqref{eq:matching-empirical-gaussian-energy} arguing as in \cref{prop:matching-heat-sharp-energy}.

We thus prove \eqref{eq:matching-lr-coupling}. Notice first that with our notation,
\begin{equation}\label{eq:matching-hilbert-sum}
    J_{\zeta_N^\sharp}
    =\frac1{\sqrt N}\sum_{i=1}^N Z_{i,N},
    \qquad
    Z_{i,N}:=\frac{V_{Y_i}^{\Lambda_N}
      -V_{X_i}^{\Lambda_N}}{\sqrt2}.
\end{equation}
The vectors $Z_{i,N}$ are i.i.d., centered, and by triangle inequality
\begin{equation}\label{eq:matching-summand-bound}
    \|Z_{i,N}\|_{L^2}\le\sqrt{2m_{\Lambda_N}}
    \qquad\text{almost surely}.
\end{equation}

The covariance of $Z_{i,N}$ is exactly the covariance of
$\sJ_{\Lambda_N}$.  Indeed, for nonzero cutoff modes $k,\ell$,
\begin{align}
    E\left[
      \widehat{\zeta_N^\sharp}(k)
      \overline{\widehat{\zeta_N^\sharp}(\ell)}
    \right]
    &=\one_{k=\ell},
    \label{eq:matching-empirical-covariance}\\
    E\left[
      \widehat{\zeta_N^\sharp}(k)
      \widehat{\zeta_N^\sharp}(\ell)
    \right]
    &=\one_{k=-\ell}.
    \label{eq:matching-empirical-pseudocovariance}
\end{align}
These are precisely the covariance and pseudocovariance relations of real cutoff white noise.  Since the covariance of the normalized sum in \eqref{eq:matching-hilbert-sum} equals the covariance of each summand, application of the Poisson multiplier shows that $Z_{i,N}$ and $\sJ_{\Lambda_N}$ have the same covariance operator on
$\mathscr H_{\Lambda_N}$.

Theorem~1.1 of \cite{Zhai18}, applied after identifying
$\mathscr H_{\Lambda_N}$ isometrically with
$\R^{d_{\Lambda_N}}$, therefore gives
\begin{align}
    \mathcal W_2^{\mathscr H_{\Lambda_N}}
    \left(
      \operatorname{Law}(J_{\zeta_N^\sharp}),
      \operatorname{Law}(\sJ_{\Lambda_N})
    \right)
    &\le
    \frac{5\sqrt{d_{\Lambda_N}}\sqrt{2m_{\Lambda_N}}
      (1+\log N)}{\sqrt N}
    \notag\\
    &\stackrel{\eqref{eq:matching-mesoscopic-scale}}{\le} C\sqrt{m_{\Lambda_N}}
      \frac{\log N}{\sqrt{\alpha_N}}
    =o(\sqrt{m_{\Lambda_N}}).
    \label{eq:matching-zhai-bound}
\end{align}
Here $\mathcal W_2^{\mathscr H_\lambda}$ denotes quadratic transportation distance for the Hilbert norm $\|\cdot\|_{L^2}$, and the final estimate uses $B>2$ in \eqref{eq:matching-mesoscopic-scale}.

Choose an optimal coupling in \eqref{eq:matching-zhai-bound}.  For
$1\le r\le2$, H\"older inequality yields
\begin{equation}\label{eq:matching-lr-coupling-low}
    \left(
      E\bra{\|J_{\zeta_N^\sharp}-\sJ_{\Lambda_N}\|_{L^q}^q}
    \right)^{1/q}
    \le
    \left(
      E\bra{\|J_{\zeta_N^\sharp}-\sJ_{\Lambda_N}\|_{L^2}^2}
    \right)^{1/2}
    =o(\sqrt{m_{\Lambda_N}}).
\end{equation}
For $q>2$, fix $s>q$.  By
\cref{lem:matching-kernel-moments},
\begin{equation}\label{eq:matching-empirical-high-moment}
    E\bra{\|J_{\zeta_N^\sharp}\|_{L^s}^s}
    \le C_s m_{\Lambda_N}^{s/2}.
\end{equation}
For every $s>0$, the Gaussian field satisfies the analogous estimate
\begin{equation}\label{eq:matching-gaussian-high-moment}
    E\bra{\|\sJ_{\Lambda_N}\|_{L^s}^s}
    =E\bra{|\sJ_{\Lambda_N}(0)|^s}
    \le C_s m_{\Lambda_N}^{s/2},
\end{equation}
because its pointwise covariance has trace $m_{\Lambda_N}$.  Consequently, under any coupling,
\begin{equation*}
    E\bra{\|J_{\zeta_N^\sharp}-\sJ_{\Lambda_N}\|_{L^s}^s}
    \le C_s m_{\Lambda_N}^{s/2}.
\end{equation*}
Choose $\theta\in(0,1)$ so that
$1/q=\theta/2+(1-\theta)/s$.  Spatial interpolation followed by H\"older's inequality in probability gives
\begin{align*}
    \left(
      E\bra{\|J_{\zeta_N^\sharp}-\sJ_{\Lambda_N}\|_{L^q}^q}
    \right)^{1/q}
    &\le
    \left(
      E\bra{\|J_{\zeta_N^\sharp}-\sJ_{\Lambda_N}\|_{L^2}^2}
    \right)^{\theta/2}
    \left(
      E\bra{\|J_{\zeta_N^\sharp}-\sJ_{\Lambda_N}\|_{L^s}^s}
    \right)^{(1-\theta)/s}\\
    &=o(\sqrt{m_{\Lambda_N}}).
\end{align*}
Together with \eqref{eq:matching-lr-coupling-low}, this proves the claimed \eqref{eq:matching-lr-coupling}.
\end{proof}

%-----------------------------------------------------------------
\subsection{Conclusion}

By \cref{thm:main},
\begin{equation}\label{eq:matching-apply-main-theorem}
    E\bra{\E_q(\xi_\lambda)}
    =\kappa_q m_\lambda^{q/2}+o(m_\lambda^{q/2})
    \qquad(\lambda\to\infty).
\end{equation}
Combining \eqref{eq:matching-energy-normalization},
\eqref{eq:matching-heat-sharp-energy},
\eqref{eq:matching-empirical-gaussian-energy}, and
\eqref{eq:matching-apply-main-theorem}, we obtain
\begin{equation}\label{eq:matching-pde-energy-asymptotic}
    E\bra{\int_{\T^2}|\nabla\varphi_N|^{p}}
    =\kappa_q\left(\frac{2m_{\Lambda_N}}N\right)^{q/2}
    +o\left(\left(\frac{\log N}{N}\right)^{q/2}\right).
\end{equation}
By \eqref{eq:linear-variance} and
\eqref{eq:matching-mesoscopic-scale},
\begin{equation*}
    \log\Lambda_N
    =\frac12\log N-\frac B2\log\log N,
    \qquad
    \frac{2m_{\Lambda_N}}N
    =\frac{\log N}{2\pi N}(1+o(1)).
\end{equation*}
Substituting this into \eqref{eq:matching-pde-energy-asymptotic} and then using \eqref{eq:matching-avt-reduction} proves
\eqref{eq:matching-exact-limit}.  

%-----------------------------------------------------------------
\subsection{The endpoint matching problem}
\label{ssec:matching-qone}

We now prove \cref{thm:matching-exact-asymptotic} for $q=1$.  Set
\begin{equation}\label{eq:matching-endpoint-scale}
    a_N:=\left(\frac{\log N}{2\pi N}\right)^{1/2}.
\end{equation}

We first establish the lower bound.  We recall that the heat semigroup is contractive for $W_1$ (see \cite{fathi2025quantitative} for quantitative versions):
\begin{equation}\label{eq:matching-heat-W1-contraction}
    W_1(P_t\mu,P_t\nu)\le W_1(\mu,\nu)
    \qquad(t>0).
\end{equation}

For the smooth densities introduced in \eqref{eq:matching-fluctuation-forcing}, Kantorovich--Rubinstein duality and \cref{lem:qone-finite-duality} give
\begin{equation}\label{eq:matching-W1-Beckmann}
    W_1(P_{t_N}\mu_N,P_{t_N}\nu_N)
    =\E_1(f_N)
    =\left(\frac2N\right)^{1/2}
      \E_1(\zeta_N^{\mathrm h}).
\end{equation}

The endpoint cases of
\cref{prop:matching-heat-sharp-energy,prop:matching-empirical-gaussian-energy},
followed by \cref{thm:qone-ultraviolet1}, yield
\begin{equation}
    E\bra{\E_1(\zeta_N^{\mathrm h})}
    =E\bra{\E_1(\zeta_N^\sharp)}+o(\sqrt{m_{\Lambda_N}})
    =E\bra{\E_1(\xi_{\Lambda_N})}+o(\sqrt{m_{\Lambda_N}})
    =\kappa_1\sqrt{m_{\Lambda_N}}
      +o(\sqrt{m_{\Lambda_N}}).
    \label{eq:matching-endpoint-smoothed-energy}
\end{equation}
  Combining
\eqref{eq:matching-heat-W1-contraction},
\eqref{eq:matching-W1-Beckmann}, and
\eqref{eq:matching-endpoint-smoothed-energy}, we obtain
\begin{equation*}
    E[W_1(\mu_N,\nu_N)]
    \ge \kappa_1\left(\frac{2m_{\Lambda_N}}N\right)^{1/2}
      +o\left(\left(\frac{m_{\Lambda_N}}N\right)^{1/2}\right).
\end{equation*}
By definition \eqref{eq:linear-variance} of $m_{\Lambda_N}$ we obtain as in the case $r>1$,
\begin{equation}\label{eq:matching-endpoint-liminf}
    \liminf_{N\to\infty}\frac{E[W_1(\mu_N,\nu_N)]}{a_N}
    \ge\kappa_1.
\end{equation}

For the upper bound, fix any $q>1$.  Since 
$W_1\le W_q$,  by H\"older inequality,
\begin{equation*}
    E[W_1(\mu_N,\nu_N)]
    \le E[W_q(\mu_N,\nu_N)]
    \le\left(E W_q^q(\mu_N,\nu_N)\right)^{1/q}.
\end{equation*}
The already proved $q>1$ case of
\eqref{eq:matching-exact-limit} therefore gives
\begin{equation*}
    \limsup_{N\to\infty}\frac{E[W_1(\mu_N,\nu_N)]}{a_N}
    \le\kappa_q^{1/q}.
\end{equation*}
Now let $q\downarrow1$.  By \cref{prop:qone-continuity},
$\kappa_q\to\kappa_1>0$, and hence
$\kappa_q^{1/q}\to\kappa_1$.  Thus
\begin{equation}\label{eq:matching-endpoint-limsup}
    \limsup_{N\to\infty}\frac{E[W_1(\mu_N,\nu_N)]}{a_N}
    \le\kappa_1.
\end{equation}
Together, \eqref{eq:matching-endpoint-liminf} and
\eqref{eq:matching-endpoint-limsup} prove
\eqref{eq:matching-exact-limit} at
$q=1$.

%%%%%%%%%%%%%%%%%%%%%%%%%%%%%%%%%%%%%%%%%%%%%%%%%%%%%%%%%%%%%%%%%%
\appendix
\section{Proofs of Technical Results}
\label[appendix]{app:appendix}

%-------------------------------------------
\subsection{Ultraviolet Cutoff Problem}

\begin{proof}[Proof of \cref{prop:finite-cutoff-variational}]
The first part of the claim is standard, we only prove  measurability of $\sH_{q,\Lambda}$.  Let
$\Ssc_q\subset L^q(\T^2;\R^2)$ be the closed subspace of solenoidal
fields.  Every admissible class is $J+\Ssc_q$ for a particular flux $J$.
Suppose that $J_n\to J$ in $L^q$, let $H_n$ minimize on $J_n+\Ssc_q$, and
let $H$ minimize on $J+\Ssc_q$.  The bound $\|H_n\|_q\le\|J_n\|_q$
makes $(H_n)$ bounded.  Every weak cluster point $\overline H$ belongs to
$J+\Ssc_q$, and the competitor $H+J_n-J$ gives
\begin{equation}\label{eq:minimizer-continuity-limsup}
    \limsup_{n\to\infty}\|H_n\|_q^q
    \le\lim_{n\to\infty}\|H+J_n-J\|_q^q
    =\|H\|_q^q.
\end{equation}
Weak lower semicontinuity and minimality of $H$ give, along every weakly
convergent subsequence,
\begin{equation*}
    \|H\|_q^q
    \le\|\overline H\|_q^q
    \le\liminf_{n\to\infty}\|H_n\|_q^q.
\end{equation*}
Together with \eqref{eq:minimizer-continuity-limsup}, uniqueness identifies
$\overline H=H$ and yields $\|H_n\|_q\to\|H\|_q$.  Uniform convexity then
upgrades weak convergence to strong convergence.  Hence the minimizing
representative depends continuously on $J$ in $L^q$.  Since
$\sJ_{\Lambda}$ is a finite-dimensional measurable function
of the Fourier coefficients of $\xi_\Lambda$, the minimizer
$\sH_{q,\Lambda}$ is strongly measurable.
\end{proof}

%\begin{proof}[Proof of \cref{lem:angular-minimization}]
% For fixed $\theta$, completing the square gives
% \begin{align*}
%  &(e_\theta+a e_\theta^\perp)^\top
%  H(e_\theta+a e_\theta^\perp)
%  \\
%  &\qquad=
%  (e_\theta^\perp)^\topHe_\theta^\perp
%  \left(a+
%  \frac{(e_\theta^\perp)^\topHe_\theta}
%  {(e_\theta^\perp)^\topHe_\theta^\perp}
%  \right)^2
%  +\frac{\det H}{(e_\theta^\perp)^\topHe_\theta^\perp}.
% \end{align*}
% The $b$-term is nonnegative.  Diagonalizing $H$, with eigenvalues
% $\lambda_1,\lambda_2>0$, yields
% \begin{equation*}
%  \int_0^\pi
%  \frac{d\theta/\pi}{\lambda_1\sin^2\theta+\lambda_2\cos^2\theta}
%  =\frac1{\sqrt{\lambda_1\lambda_2}}.
% \end{equation*}
% Multiplication by $\det H$ proves the result.
% \end{proof}

\begin{proof}[Proof of \cref{lem:predictable-density}]
Write
\begin{equation*}
 A_t
 :=\int_0^t\int_0^\pi
 \bigl(|a_s(\theta)|^2+|b_s(\theta)|^2\bigr)
 \frac{d\theta}{\pi}\,ds.
\end{equation*}
This is a continuous adapted increasing process and is therefore
predictable.  For $N\ge1$, set $(a^{[N]}_s,b^{[N]}_s)
 :=\one_{\{A_s<N^2\}}(a_s,b_s)$. The truncated pair is predictable, its total square function is at most $N^2$, and the absolute continuity of $A$ gives
\begin{equation}\label{eq:square-function-stopping}
 \int_0^T\int_0^\pi
 \left(
 |a_s-a_s^{[N]}|^2+|b_s-b_s^{[N]}|^2
 \right)\frac{d\theta}{\pi}\,ds
 =(A_T-N^2)_+.
\end{equation}
Since $E[A_T^{r/2}]<\infty$, dominated convergence in
\eqref{eq:square-function-stopping} yields as $N\to \infty$
\begin{equation}\label{eq:hardy-stopping-convergence}
 \|(a,b)-(a^{[N]},b^{[N]})\|_{\H_T^r}
 \longrightarrow0.
\end{equation}
Let $\Pi_M$ be the metric projection of $\R^2$ onto the closed ball of radius $M$, and define
\begin{equation*}
 (a^{[N,M]},b^{[N,M]})
 :=\Pi_M(a^{[N]},b^{[N]}).
\end{equation*}
The pair is predictable and bounded by $M$.  Its difference from
$(a^{[N]},b^{[N]})$ converges pointwise to zero as $M\to\infty$, while its
squared integral is bounded by $N^2$.  Hence another application of dominated
convergence gives
\begin{equation}\label{eq:hardy-amplitude-convergence}
 \|(a^{[N]},b^{[N]})-(a^{[N,M]},b^{[N,M]})\|_{\H_T^r}
 \longrightarrow0
 \qquad(M\to\infty).
\end{equation}

It remains to approximate a bounded predictable pair.  We complete the predictable $\sigma$-algebra with respect to $dP\,ds$, and complete its product
with the Borel $\sigma$-algebra on $[0,\pi]$ with respect to
$dP\,ds\,d\theta/\pi$.  Modulo null sets, this product $\sigma$-algebra is
generated by rectangles
\begin{equation*}
 C\times(u,v]\times B,
 \qquad C\in\F_u,
 \quad B\subset[0,\pi]\text{ Borel}.
\end{equation*}
Indicators of these rectangles are elementary predictable controls.  A
monotone-class argument therefore shows that elementary predictable controls
are dense in
\begin{equation*}
 L^2\left(
 \Omega\times(0,T]\times[0,\pi],
 dP\,ds\,\frac{d\theta}{\pi}
 \right).
\end{equation*}
Consequently, if a predictable pair $\gamma$ is bounded by $M$, there are
elementary predictable pairs $\gamma_k$ converging to $\gamma$ in this
$L^2$ space.  After refining the finitely many time and angular sets in each
$\gamma_k$ and applying $\Pi_M$ to its adapted coefficient on each resulting
cell, we may also assume $|\gamma_k|\le M$.  The projection is
$1$-Lipschitz, so the $L^2$ convergence is preserved.

Set
\begin{equation*}
 D_k
 :=\int_0^T\int_0^\pi
 |\gamma_k(s,\theta)-\gamma(s,\theta)|^2
 \frac{d\theta}{\pi}\,ds.
\end{equation*}
Then $E \bra{D_k}\to0$ and $0\le D_k\le4M^2T$.  If $1<r\le2$, Jensen's
inequality gives
\begin{equation*}
 E[D_k^{r/2}]\le(E \bra{D_k})^{r/2}\longrightarrow0.
\end{equation*}
If $r>2$, then
\begin{equation*}
 E[D_k^{r/2}]
 \le(4M^2T)^{r/2-1}E\bra{ D_k}
 \longrightarrow0.
\end{equation*}
Thus bounded elementary predictable pairs are dense among bounded
predictable pairs in $\H_T^r$, with the same essential bound.
If, for example, the target has $b=0$, replace an approximating pair
$(a_k,b_k)$ by $(a_k,0)$; this can only decrease its distance from the
target.  Combining this with
\eqref{eq:hardy-stopping-convergence} and
\eqref{eq:hardy-amplitude-convergence}, and taking a diagonal sequence,
proves the asserted density.

For an elementary pair, the process $M_t:=\I_t(a,b)$ is a continuous $\R^2$-valued martingale.  Its total scalar quadratic
variation is
\begin{equation*}
 [M]_T^{\mathrm{tr}}:=\sum_{i=1}^2[M^i]_T
 =\int_0^T\int_0^\pi
 \bigl(|a_s(\theta)|^2+|b_s(\theta)|^2\bigr)
 \frac{d\theta}{\pi}\,ds.
\end{equation*}
The Burkholder--Davis--Gundy inequality gives
\eqref{eq:bdg-forward-controls}.  Density then defines
\eqref{eq:control-stochastic-integral} for every pair in
$\H_T^r$ and preserves the same estimate.

Finally, suppose that $M$ is initially square integrable and $M_T\in L^r$.
The reverse Burkholder--Davis--Gundy inequality and Doob's maximal inequality
imply
\begin{equation*}
 \|([M]_T^{\mathrm{tr}})^{1/2}\|_{L^r}
 \le C_r\|\sup_{t\le T}|M_t|\|_{L^r}
 \le C_r\|M_T\|_{L^r},
\end{equation*}
which is exactly \eqref{eq:bdg-reverse-controls}.
\end{proof}

%-------------------------------------------
\subsection{ODE Study}\label{ssec:ode}

\begin{lemma}\label{lem:h}
    For $q>1$, the singular initial-value problem
\begin{equation}\label{eq:h-ode}
    \begin{cases}
        h'(s)=s\pa{q-1-h(s)^2}-s^{-1}h(s)\pa{h(s)^2-1}, &s>0,\\
        h(s)=1+\frac14\pa{q-2}s^2+O(s^4),
        &\text{as }s\downarrow0,
    \end{cases}
\end{equation}
has a unique global positive solution $h=h_q$. Moreover, the solution satisfies
\begin{equation}\label{eq:h-range}
    \min\set{1,\sqrt{q-1}}
    \le h_q(s)\le
    \max\set{1,\sqrt{q-1}},
\end{equation}
with strict inequalities for $s>0$ when $q\ne2$, and
\begin{equation}\label{eq:h-infinity}
    h_q(s)
    =\sqrt{q-1}-\frac{q-2}{2s^2}+O_q(s^{-4})
    \quad\text{as }s\to\infty.
\end{equation}
\end{lemma}

\begin{proof}
For $q=2$ one has $h_2\equiv1$. Assume $q\ne2$ and put $y_0=(q-2)/4$. Substituting
\begin{equation*}
    h(s)=1+s^2\pa{y_0+z(s)}
\end{equation*}
into \eqref{eq:h-ode} gives
\begin{equation}\label{eq:z-fixed-point}
    z(s)=s^{-4}\int_0^s r^5\F_q(r,z(r))\,dr,
\end{equation}
where
\begin{equation*}
    \F_q(r,z)
    =-2(y_0+z)-3(y_0+z)^2
     -r^2\pa{(y_0+z)^2+(y_0+z)^3}.
\end{equation*}
On a fixed ball of $C([0,s_0])$, the right-hand side of
\eqref{eq:z-fixed-point} has norm $O_q(s_0^2)$ and Lipschitz constant
$O_q(s_0^2)$.  It is a contraction when $s_0$ is sufficiently small.  This proves local existence and uniqueness and gives
\begin{equation}\label{eq:h-origin-expansion}
    h(s)=1+\frac{q-2}{4}s^2+O_q(s^4).
\end{equation}

Set $a=\sqrt{q-1}$ and denote the right-hand side of \eqref{eq:h-ode} by
$F(s,h)$.  At the two boundary values,
\begin{equation}\label{eq:h-boundary-vector-field}
    F(s,1)=s(q-2),
    \qquad
    F(s,a)=-\frac{a(q-2)}s.
\end{equation}
Suppose first that $q>2$.  The expansion \eqref{eq:h-origin-expansion}
places $h$ in $(1,a)$ for small positive $s$.  At a first contact with the
lower boundary $1$, one would have $h'\le0$, contrary to
$F(s,1)>0$; at a first contact with the upper boundary $a$, one would have
$h'\ge0$, contrary to $F(s,a)<0$.  Thus $1<h(s)<a$ for every $s>0$.
When $1<q<2$, the same first-contact argument, with the interval $(a,1)$ and
the signs in \eqref{eq:h-boundary-vector-field} reversed, gives
$a<h(s)<1$.  This proves \eqref{eq:h-range} and its strict form.  Since the
solution remains in a compact subinterval of $(0,\infty)$, the standard
continuation criterion for the nonsingular equation on $[s_0,\infty)$ gives
global existence.

Let $\delta=h-a$.  The equation can be written exactly as
\begin{equation}\label{eq:delta-exact}
    \delta'
    =-s(h+a)\delta-\frac{h(h^2-1)}s.
\end{equation}
The range bound implies, for $s\ge1$,
\begin{equation*}
    D^+|\delta(s)|
    \le-c_qs|\delta(s)|+\frac{C_q}s.
\end{equation*}
The scalar Gronwall inequality therefore gives
\begin{equation}\label{eq:delta-gronwall}
    |\delta(s)|
    \le \e^{-c_q(s^2-1)/2}|\delta(1)|
      +C_q\int_1^s\e^{-c_q(s^2-r^2)/2}\frac{dr}{r}.
\end{equation}
For $s\ge2$, split the integral at $s/2$.  On $[1,s/2]$ one has
$s^2-r^2\ge3s^2/4$.  On $[s/2,s]$ one has $r^{-1}\le2s^{-1}$ and
$s^2-r^2\ge s(s-r)$.  Hence
\begin{equation*}
    \int_1^s\e^{-c_q(s^2-r^2)/2}\frac{dr}{r}
    \le \e^{-3c_qs^2/8}\log s
       +\frac2s\int_0^\infty\e^{-c_qsu/2}\,du
    \le\frac{C_q}{s^2}.
\end{equation*}
After increasing the constant on $[1,2]$, \eqref{eq:delta-gronwall} yields
$\delta(s)=O_q(s^{-2})$, and in particular $h(s)\to a$.

Expanding \eqref{eq:delta-exact} around $a$ gives
\begin{equation*}
    \delta'
    =-2as\delta-\frac{a(q-2)}s
     +O_q\pa{s\delta^2+\frac{|\delta|}{s}}.
\end{equation*}
Define
\begin{equation*}
    \rho(s):=\delta(s)+\frac{q-2}{2s^2}.
\end{equation*}
Since $\delta=O_q(s^{-2})$,
\begin{equation*}
    \rho'
    =-2as\rho
     +O_q\pa{\frac{|\rho|}{s}+s\rho^2+s^{-3}}
    =-2as\rho+O_q(s^{-3}).
\end{equation*}
For a sufficiently large fixed $S$, variation of constants gives
\begin{equation}\label{eq:rho-variation}
    \rho(s)
    =\e^{-a(s^2-S^2)}\rho(S)
     +\int_S^s\e^{-a(s^2-r^2)}O_q(r^{-3})\,dr.
\end{equation}
Splitting the last integral at $s/2$ as above, and using
$r^{-3}\le8s^{-3}$ on $[s/2,s]$, gives an $O_q(s^{-4})$ bound.  Thus
\eqref{eq:rho-variation} proves \eqref{eq:h-infinity}.  Near zero,
$h_q(s)^2-1=O_q(s^2)$; at infinity,
$h_q(s)^2-(q-1)=O_q(s^{-2})$.
\end{proof}

\begin{proof}[Proof of \cref{lem:phiqexists}]
We shall in fact prove the following representation of the solution,
\begin{gather}
    \label{eq:kappa-ode-intro}
    \kappa_q^{\mathrm{ode}}
    :=\exp\set{
       -\int_0^1\frac{h_q(s)^2-1}{s}\,ds
       -\int_1^\infty\frac{h_q(s)^2-(q-1)}{s}\,ds
    },\\
    g_q(s)=q\kappa_q^{\mathrm{ode}}
      \exp\set{\int_0^s\frac{h_q(r)^2-1}{r}\,dr},
    \quad
    \phi_q(s) =\kappa_q^{\mathrm{ode}}+\int_0^s r g_q(r)\,dr,
\end{gather}
and that $\phi_q$ is strictly convex and
\begin{gather}
    \label{eq:profile-hessian}
    \phi_q'(s)=s g_q(s),
    \qquad\phi_q''(s)=h_q(s)^2g_q(s),\\
    \label{eq:profile-ode}
    q\phi_q(s)-s\phi_q'(s)
    =h_q(s)g_q(s)
    =\sqrt{\phi_q''(s)\frac{\phi_q'(s)}s},
\end{gather}
where the value at $s=0$ is interpreted by continuity.

For $q=2$, since $h_2\equiv1$, $\kappa_2^{\mathrm{ode}}=1$ and
$\phi_2(s)=1+s^2$. From now on assume that $q\neq 2$. The estimates on $h_q$ prove the absolute convergence of both integrals in \eqref{eq:kappa-ode-intro}. By definition,
\begin{equation*}
    \frac{g_q'(s)}{g_q(s)}=\frac{h_q(s)^2-1}{s},
\end{equation*}
which gives \eqref{eq:profile-hessian}.  Let
\begin{equation*}
    B(s):=q\phi_q(s)-s^2g_q(s)-h_q(s)g_q(s).
\end{equation*}
Since $g_q(0)=q\kappa_q^{\mathrm{ode}}$, one has $B(0)=0$.  Direct
differentiation using \eqref{eq:h-ode} gives $B'=0$, and hence
\eqref{eq:profile-ode}.  The definition of $\kappa_q^{\mathrm{ode}}$ and
\eqref{eq:h-infinity} imply
\begin{equation}\label{eq:g-normalization}
    \lim_{s\to\infty}\frac{g_q(s)}{q s^{q-2}}=1.
\end{equation}
Indeed, the logarithm of the ratio equals
\begin{equation*}
    -\int_s^\infty\frac{h_q(r)^2-(q-1)}r\,dr,
\end{equation*}
which tends to zero.  Integrating $\phi_q'(s)=s g_q(s)$ proves the final
condition in \eqref{eq:profile-boundary}.  Strict convexity follows from
$g_q>0$ and $h_q>0$.

It remains to prove uniqueness.  Let $\phi$ be any strictly convex $C^2$
solution with the stated boundary conditions, put $g=\phi'/s$ for $s>0$,
and define
\begin{equation*}
    h(s):=\frac{q\phi(s)-s\phi'(s)}{g(s)}.
\end{equation*}
The positive square root in \eqref{eq:profile-ode} gives
$\phi''=h^2g$, whence $g'/g=(h^2-1)/s$; differentiating the definition of
$h$ shows that it solves \eqref{eq:h-ode}.  The conditions at zero imply
$h(s)\to1$.  If $v=h-1$, then
\begin{equation}\label{eq:v-integrated-identity}
    (s^2v)'
    =s^3(q-2)-2s^3v-s^3v^2-3sv^2-sv^3.
\end{equation}
Since $s^2v(s)\to0$, integration from zero gives
\begin{align*}
    s^2v(s)
    =\int_0^s\bra{
       r^3(q-2)-2r^3v(r)-r^3v(r)^2
       -3rv(r)^2-rv(r)^3
    }\,dr.
\end{align*}
For $M(s):=\sup_{0<r\le s}|v(r)|$, division by $t^2$ and taking the
supremum over $0<t\le s$ give, for small $s$,
\begin{equation*}
    M(s)\le C_qs^2+C_qs^2M(s)+C_qM(s)^2+C_qM(s)^3.
\end{equation*}
Because $M(s)\to0$, the last three terms are absorbed after decreasing
$s$, and $M(s)\le C_qs^2$.  Substitution into the integrated identity
\eqref{eq:v-integrated-identity} then gives
\begin{equation*}
    \lim_{s\downarrow0}\frac{v(s)}{s^2}=\frac{q-2}{4}.
\end{equation*}
The local uniqueness obtained from \eqref{eq:z-fixed-point} therefore gives
$h=h_q$.

The equation $g'/g=(h_q^2-1)/s$ determines $g$ up to a positive
multiplicative constant.  By \eqref{eq:h-infinity}, every such solution has
$g(s)\sim C s^{q-2}$ for some $C>0$, and consequently
$\phi(s)/s^q\to C/q$.  Thus the normalization
$\phi(s)/s^q\to1$ is equivalent to $C=q$, which is precisely
\eqref{eq:g-normalization}.  Hence $g=g_q$; the identity
$q\phi-s\phi'=h_qg_q$ then gives $\phi=\phi_q$.
\end{proof}

\begin{lemma}\label{lem:bellman-terminal}
For every $r>0$ and $x\in\R^2$,
\begin{equation}\label{eq:bellman-terminal-rate}
 0\le u_q(r,x)-|x|^q
 \le C_q
 \begin{cases}
  r^{q/2},&1<q\le2,\\
  r|x|^{q-2}+r^{q/2},&q>2.
 \end{cases}
\end{equation}
In particular, $u_q(r,\cdot)\to|\cdot|^q$ locally uniformly as
$r\downarrow0$.  If $\sY_r$, $0<r\leq 1$ is a $L^q$-bounded random field,
\begin{equation}\label{eq:random-terminal-rate}
 E\big[u_q(r,\sY_r)-|\sY_r|^q\big]\longrightarrow0.
\end{equation}
\end{lemma}

\begin{proof}
Put $s=|x|/\sqrt r$.  Differentiating the self-similar formula and using
\eqref{eq:profile-ode} gives
\begin{equation}\label{eq:bellman-time-derivative}
 \partial_r u_q(r,x)
 =\frac12r^{q/2-1}\big(q\phi_q(s)-s\phi_q'(s)\big)
 =\frac12r^{q/2-1}h_q(s)g_q(s)>0.
\end{equation}
The range bound for $h_q$, continuity of $g_q$, and
\eqref{eq:g-normalization} imply
\begin{equation}\label{eq:hg-growth}
 h_q(s)g_q(s)
 \le C_q
 \begin{cases}
  1,&1<q\le2,\\
  1+s^{q-2},&q>2.
 \end{cases}
\end{equation}
For fixed $x$, integrate \eqref{eq:bellman-time-derivative} from zero to
$r$, using $u_q(t,x)\to|x|^q$ as $t\downarrow0$.  The bound
\eqref{eq:hg-growth} is integrable at zero and gives exactly
\eqref{eq:bellman-terminal-rate}.  Local uniform convergence follows from
the same estimate.  Finally, for $q>2$, a uniform $q$-moment bounds the
$(q-2)$-moment uniformly; for $q\le2$, the right-hand side of
\eqref{eq:bellman-terminal-rate} is deterministic.  Taking expectations
proves \eqref{eq:random-terminal-rate}.
\end{proof}

\begin{lemma}
\label{lem:feedback-regularity}
Let $\beta_q(s):=h_q(s)^2-1$.  For $r>0$, $x\ne0$, and
$\theta\in[0,\pi]$, define
\begin{equation}\label{eq:feedback-definition}
 a_q^\star(r,x,\theta)
 :=-
 \frac{\beta_q(|x|/\sqrt r)
 (e_\theta^\perp\cdot\widehat x)(e_\theta\cdot\widehat x)}
 {1+\beta_q(|x|/\sqrt r)
 (e_\theta^\perp\cdot\widehat x)^2},
 \qquad \widehat x:=\frac{x}{|x|},
\end{equation}
and set $a_q^\star(r,0,\theta)=0$.  There is $C_q<\infty$ such that
\begin{align}
 \sup_{r>0,x,\theta}|a_q^\star(r,x,\theta)|&\le C_q,
 \label{eq:feedback-uniform-bound}\\
 \left(\int_0^\pi
 |a_q^\star(r,x,\theta)-a_q^\star(r,y,\theta)|^2
 \frac{d\theta}{\pi}\right)^{1/2}
 &\le C_qr^{-1/2}|x-y|.
 \label{eq:feedback-lipschitz}
\end{align}
Consequently, for every $\tau>\varepsilon>0$, the equation
\begin{equation}\label{eq:feedback-sde}
 \sX_t=x+
 \int_0^t\int_0^\pi
 \big(e_\theta+a_q^\star(\tau-s,\sX_s,\theta)e_\theta^\perp\big)
 \,\sw^1(ds,d\theta),
 \qquad 0\le t\le\tau-\varepsilon,
\end{equation}
has a unique strong solution.  For every $m>0$,
\begin{equation}\label{eq:feedback-moment-bound}
 \sup_{0<\varepsilon<\tau}
 E\left[\sup_{t\le\tau-\varepsilon}|\sX_t|^m\right]
 \le C_{m,q}\big(|x|^m+\tau^{m/2}\big).
\end{equation}
\end{lemma}

\begin{proof}
The origin expansion \eqref{eq:h-origin-expansion}, inserted into
\eqref{eq:h-ode}, gives
\begin{equation*}
 \beta_q(s)=O_q(s^2),
 \qquad
 \beta_q'(s)=O_q(s)
 \qquad(s\downarrow0).
\end{equation*}
At infinity, put $a=\sqrt{q-1}$.  By
\eqref{eq:h-infinity},
\begin{equation*}
 h_q(s)=a-\frac{q-2}{2s^2}+O_q(s^{-4}),
 \qquad
 \beta_q(s)=q-2+O_q(s^{-2}).
\end{equation*}
Inserting the first expansion into the two terms on the right-hand side of
\eqref{eq:h-ode} gives separately
\begin{align*}
 s\big(q-1-h_q(s)^2\big)
 &=\frac{a(q-2)}s+O_q(s^{-3}),\\
 \frac{h_q(s)(h_q(s)^2-1)}s
 &=\frac{a(q-2)}s+O_q(s^{-3}).
\end{align*}
The leading terms cancel, so $h_q'(s)=O_q(s^{-3})$, without any
differentiation of the remainder in \eqref{eq:h-infinity}.  Consequently,
\begin{equation*}
 \beta_q'(s)=O_q(s^{-3})
 \qquad(s\to\infty).
\end{equation*}
Together with smoothness on compact subintervals of $(0,\infty)$, this
proves
\begin{equation}\label{eq:beta-derivative-bound}
 \sup_{s>0}
 \left(|\beta_q'(s)|+\frac{|\beta_q(s)|}{s}\right)<\infty.
\end{equation}
The range bound \eqref{eq:h-range} gives, for $t\in[0,1]$,
\begin{equation}\label{eq:feedback-denominator}
 1+\beta_q(s)t
 =(1-t)+t h_q(s)^2
 \ge \min\{1,q-1\}>0.
\end{equation}
This proves \eqref{eq:feedback-uniform-bound}.

For $y=s\omega\ne0$, with $\omega\in S^1$, write
\begin{equation*}
 F_\theta(y)
 :=-
 \frac{\beta_q(s)(e_\theta^\perp\cdot\omega)
 (e_\theta\cdot\omega)}
 {1+\beta_q(s)(e_\theta^\perp\cdot\omega)^2}.
\end{equation*}
Direct differentiation, using \eqref{eq:feedback-denominator}, gives
uniformly in $\theta$ and $\omega$
\begin{equation*}
 |\partial_sF_\theta(s\omega)|\le C_q|\beta_q'(s)|,
 \qquad
 |\nabla_{S^1}F_\theta(s\omega)|\le C_q|\beta_q(s)|.
\end{equation*}
Hence, on $\R^2\setminus\{0\}$,
\begin{equation*}
 |D F_\theta(y)|
 \le C_q\left(|\beta_q'(|y|)|+
 \frac{|\beta_q(|y|)|}{|y|}\right).
\end{equation*}
The right-hand side is uniformly bounded by
\eqref{eq:beta-derivative-bound} and tends to zero as $y\to0$.  Also
$|F_\theta(y)|\le C_q|y|^2$ near zero.  Thus $F_\theta$ extends to a
$C^1$ function on $\R^2$, with $D F_\theta(0)=0$, and is globally
Lipschitz with a constant independent of $\theta$.  Since
\begin{equation*}
 a_q^\star(r,x,\theta)=F_\theta(x/\sqrt r),
\end{equation*}
we obtain the stronger pointwise version of
\eqref{eq:feedback-lipschitz}, and hence that estimate after integration in
$\theta$.

Let $\Theta=[0,\pi]$.
Let $\HH_\Theta:=L^2(\Theta,d\theta/\pi)$.  The white noise
$\sw^1$ induces a cylindrical Brownian motion on $\HH_\Theta$.  Let
$\Sigma_r(x):\HH_\Theta\to\R^2$ be the Hilbert--Schmidt operator with
kernel
\begin{equation*}
 \theta\longmapsto
 e_\theta+a_q^\star(r,x,\theta)e_\theta^\perp.
\end{equation*}
Equations \eqref{eq:feedback-uniform-bound} and
\eqref{eq:feedback-lipschitz} imply
\begin{equation*}
 \|\Sigma_r(x)\|_{\mathrm{HS}}\le C_q,
 \qquad
 \|\Sigma_r(x)-\Sigma_r(y)\|_{\mathrm{HS}}
 \le C_qr^{-1/2}|x-y|.
\end{equation*}
For $0\le s\le\tau-\varepsilon$, the coefficient
$\Sigma_{\tau-s}$ is therefore globally Lipschitz uniformly in time and
uniformly bounded.  Picard iteration, using It\^o's isometry and Gronwall's
inequality, gives pathwise uniqueness and a unique strong solution of
\eqref{eq:feedback-sde}.  Finally, the Burkholder--Davis--Gundy inequality
and the uniform Hilbert--Schmidt bound give \eqref{eq:feedback-moment-bound}
for $m\ge2$; the remaining positive exponents follow from Jensen's
inequality.
\end{proof}

\begin{proof}[Proof of \cref{lem:ucontrol}]
For a radial function, the eigenvalues of the Hessian are the radial and tangential second derivatives.  By \eqref{eq:profile-hessian}, the eigenvalues of $D_x^2u_q$ are
\begin{equation}\label{eigenD2u}
    \lambda_{\rm rad}=\tau^{q/2-1}h_q(s)^2g_q(s),
    \qquad \textrm{and } \qquad 
    \lambda_{\rm tan}=\tau^{q/2-1}g_q(s),
    \qquad s=|x|/\sqrt\tau.
\end{equation}
They are positive, and \eqref{eq:profile-ode} gives \eqref{eq:parabolic-monge-ampere}.  \cref{eq:h-origin-expansion} gives $C^{1,2}$ regularity at $x=0$. \Cref{lem:bellman-terminal} gives the terminal condition, locally uniformly in $x$.  We shall also use
\begin{equation}\label{eq:bellman-growth}
    |x|^q\le u_q(\tau,x)
    \le C_q\pa{|x|^q+\tau^{q/2}}.
\end{equation}
The lower bound follows from \eqref{eq:bellman-time-derivative}.  For $q>2$, the upper bound follows from \eqref{eq:bellman-terminal-rate} and Young's inequality $\tau|x|^{q-2}\le C_q(|x|^q+\tau^{q/2})$; for $q\le2$, it follows directly from the same terminal estimate.  The formulas for the profile also
give, for every $0<\varepsilon<\tau<\infty$,
\begin{align}
    |\nabla_xu_q(r,x)|
    &\le C_{q,\varepsilon,\tau}\pa{1+|x|^{q-1}},
    \label{eq:bellman-gradient-growth}\\
    |\partial_r u_q(r,x)|+\|D_x^2u_q(r,x)\|
    &\le C_{q,\varepsilon,\tau}\pa{1+|x|^q},
    \qquad \varepsilon\le r\le\tau.
    \label{eq:bellman-derivative-growth}
\end{align}
Indeed, these estimates follow from \eqref{eq:profile-hessian},
\eqref{eq:bellman-time-derivative}, and the large-$s$ behavior
$g_q(s)\sim q s^{q-2}$.

Fix $M<\infty$ and $(a,b)\in\A_M(\tau)$, and let $\sX$ be the process in \eqref{eq:value-function}.  Its diffusion coefficient has uniformly bounded Hilbert--Schmidt norm, with a bound depending only on $M$.  The Burkholder--Davis--Gundy inequality therefore gives, for every $m>0$,
\begin{equation}\label{eq:bounded-control-moments}
    E\bra{\sup_{s\le\tau}|\sX_s|^m}<\infty.
\end{equation}
Set $H:=D_x^2u_q(\tau-s,\sX_s)$. For $0<\varepsilon<\tau$, applying It\^o's formula to $M_s:=u_q(\tau-s,\sX_s)$ on $[0,\tau-\varepsilon]$, initially stopped when $|\sX|$ reaches a radius $R$, 
we have 
\begin{equation}\label{ito}
    dM_s= \mu(s,\sX_s) ds + D_x u_q(\tau-s,\sX_s)  \cdot d\sX_s,
\end{equation}
with drift,
\begin{align*}
    \mu=-\partial_\tau u_q
    +\frac12\int_0^\pi
    \left(
        (e_\theta+a_s e_\theta^\perp)^\top
        H(e_\theta+a_s e_\theta^\perp)+b_s^2(e_\theta^\perp)^\top He_\theta^\perp
    \right)\frac{d\theta}{\pi}.
\end{align*}
By \eqref{eq:parabolic-monge-ampere}, we may rewrite this as 
\begin{equation}\label{drift}
\mu=\frac{1}{2}\left(-\sqrt{\det H}
    +\int_0^\pi
    \left(
        (e_\theta+a_s e_\theta^\perp)^\top
        H(e_\theta+a_s e_\theta^\perp)+b_s^2(e_\theta^\perp)^\top He_\theta^\perp
    \right)\frac{d\theta}{\pi}\right).
\end{equation}
Using \eqref{eq:angular-minimization} we see that  this drift is nonnegative.  For each $s$, let $\Sigma_s^{a,b}:\HH_\Theta\oplus\HH_\Theta\to\R^2$ be the Hilbert--Schmidt operator
\begin{equation*}
    \Sigma_s^{a,b}(f,g)
    :=\int_0^\pi\left[
        \big(e_\theta+a_s(\theta)e_\theta^\perp\big)f(\theta)
        +b_s(\theta)e_\theta^\perp g(\theta)
    \right]\frac{d\theta}{\pi}.
\end{equation*}
On the region $\tau-s\ge\varepsilon$,
\eqref{eq:bellman-gradient-growth} and
\eqref{eq:bounded-control-moments} imply
\begin{equation}\label{eq:verification-martingale-integrability}
    E\bra{\int_0^{\tau-\varepsilon}
    |\nabla_xu_q(\tau-s,\sX_s)|^2
    \|\Sigma_s^{a,b}\|_{\mathrm{HS}}^2\,ds}<\infty.
\end{equation}
Thus the stochastic integral in the unstopped It\^o formula is a true
martingale.
Equations \eqref{eq:bellman-growth}, \eqref{eq:bellman-derivative-growth}, and
\eqref{eq:bounded-control-moments} also give the uniform integrability needed
to let $R\to\infty$.  Integrating \eqref{ito} and taking the expectation  yields, as the expectation of the martingale part vanishes,
\begin{equation}\label{eq:verification-lower-epsilon}
    E \bra{u_q(\varepsilon,\sX_{\tau-\varepsilon})}
    \ge u_q(\tau,x).
\end{equation}

The same bracket estimate on the last time interval gives, for every $m>0$,
\begin{equation}\label{eq:last-interval-convergence}
    \|\sX_\tau-\sX_{\tau-\varepsilon}\|_{L^m}
    \le C_{m,M,\tau}\sqrt\varepsilon.
\end{equation}
By \eqref{eq:last-interval-convergence},
$E\bra{|\sX_{\tau-\varepsilon}|^q}\to E|\sX_\tau|^q$.  Moreover,
\eqref{eq:bounded-control-moments} and
\cref{lem:bellman-terminal} give
\begin{equation*}
    E\bra{u_q(\varepsilon,\sX_{\tau-\varepsilon})
    -|\sX_{\tau-\varepsilon}|^q}\longrightarrow0.
\end{equation*}
Letting $\varepsilon\downarrow0$ in
\eqref{eq:verification-lower-epsilon} therefore yields
\begin{equation}\label{eq:verification-lower}
    E\bra{|\sX_\tau|^q}\ge u_q(\tau,x).
\end{equation}
Since the map $\sY\mapsto E\bra{|\sY|^q}$ is continuous on $L^q$, the same
inequality holds on the terminal-value $L^q$ closure of bounded predictable
controls.

We prove the reverse inequality by an explicit minimizing sequence.  For
remaining time $r>0$, the Hessian is
\begin{equation*}
    D_x^2u_q(r,x)
    =r^{q/2-1}g_q(|x|/\sqrt r)
    \bra{\id+\beta_q(|x|/\sqrt r)
    \widehat x\otimes\widehat x}
\end{equation*}
for $x\ne0$, with its isotropic continuous extension at zero.  Arguing as in the proof of \eqref{eq:angular-minimization} shows that the minimizing
first-channel control is precisely the feedback defined in
\eqref{eq:feedback-definition}, and that the minimizing second-channel
control is zero.  By \cref{lem:feedback-regularity}, for every
$\varepsilon>0$ the feedback SDE on $[0,\tau-\varepsilon]$ has a unique
strong solution $\sX^\varepsilon$, the controls in
\eqref{eq:epsilon-feedback} have a deterministic bound independent of
$\varepsilon$, and the moments in \eqref{eq:feedback-moment-bound} are
uniform.

For this feedback, equality holds in the angular minimization.  The same
localized It\^o argument as above, now with zero drift, gives
\begin{equation}\label{eq:feedback-identity}
    E\bra{ u_q(\varepsilon,\sX^\varepsilon_{\tau-\varepsilon})}
    =u_q(\tau,x).
\end{equation}
Use zero control on the last interval.  Conditional on
$\sX^\varepsilon_{\tau-\varepsilon}=y$, the last increment is a centered
Gaussian vector of covariance $\varepsilon\id/2$.  If $P_\varepsilon$
denotes the corresponding heat semigroup, the lower bound
\eqref{eq:verification-lower}, applied on an interval of length
$\varepsilon$, gives
\begin{equation*}
    |y|^q\le u_q(\varepsilon,y)
    \le P_\varepsilon(|\cdot|^q)(y).
\end{equation*}
Since
\begin{align*}
    |y+z|^q-|y|^q-q|y|^{q-2}y\cdot z
    &\le C_q|z|^q,
    &&1<q\le2,\\
    |y+z|^q-|y|^q-q|y|^{q-2}y\cdot z
    &\le C_q\pa{|y|^{q-2}|z|^2+|z|^q},
    &&q>2,
\end{align*}
with the linear term interpreted as zero at $y=0$, applying the heat-flow gives as the
linear term vanishes,
\begin{equation}\label{eq:heat-remainder}
    0\le P_\varepsilon(|\cdot|^q)(y)-|y|^q
    \le C_q
    \begin{cases}
        \varepsilon^{q/2},&1<q\le2,\\
        \varepsilon |y|^{q-2}+\varepsilon^{q/2},&q>2.
    \end{cases}
\end{equation}
Using \eqref{eq:feedback-identity}, \eqref{eq:heat-remainder}, conditioning at
time $\tau-\varepsilon$, and the uniform moment bound, we obtain
\begin{align*}
    0
    \le E\bra{|\sX^\varepsilon_\tau|^q}-u_q(\tau,x)
    =E\bra{
    P_\varepsilon(|\cdot|^q)
    (\sX^\varepsilon_{\tau-\varepsilon})
    -u_q(\varepsilon,\sX^\varepsilon_{\tau-\varepsilon})}
    \longrightarrow0.
\end{align*}
Thus bounded predictable feedbacks form a minimizing sequence.

Use now, for $\varepsilon\in(0,\tau)$, on $[0,\tau-\varepsilon]$ the feedback
\begin{equation}\label{eq:epsilon-feedback}
    a_s^{\varepsilon}(\theta)
    =a_q^\star(\tau-s,\sX_s^{\varepsilon},\theta),
    \qquad b_s^{\varepsilon}=0,
\end{equation}
and use $a_s^{\varepsilon}=b_s^{\varepsilon}=0$ on the last interval. We shall prove that the corresponding costs converge to $u_q(\tau,x)$ as $\varepsilon\downarrow0$.

Choose
$\varepsilon_j\downarrow0$.  By \eqref{eq:feedback-uniform-bound}, all
feedback pairs $(a^{\varepsilon_j},0)$ have a common deterministic bound.
Apply \cref{lem:predictable-density} on $[0,\tau-\varepsilon_j]$, extend the
approximants by zero on the last interval, and choose a bounded elementary
control $\widetilde a^j$ satisfying
\begin{equation*}
    \|(\widetilde a^j-a^{\varepsilon_j},0)\|_{\H_\tau^q}
    \le j^{-1}.
\end{equation*}
The approximants may be chosen with the same common bound.  If
$\widetilde{\sX}^j_\tau$ is the corresponding terminal state, then
\eqref{eq:bdg-forward-controls} gives
\begin{equation*}
    \|\widetilde{\sX}^j_\tau-
    \sX^{\varepsilon_j}_\tau\|_{L^q}
    \le C_qj^{-1}.
\end{equation*}
The two families are uniformly bounded in $L^q$.  Hence
\begin{equation*}
    \left|E\bra{|\widetilde{\sX}^j_\tau|^q}
    -E\bra{|\sX^{\varepsilon_j}_\tau|^q}\right|
    \longrightarrow0,
\end{equation*}
for instance by
$|\,|z|^q-|w|^q|\le C_q(|z|^{q-1}+|w|^{q-1})|z-w|$ and H\"older's
inequality.  This proves \eqref{eq:value-function} for bounded elementary
controls and for their terminal-value closure.

Finally, the constant $c$ in \eqref{eq:control-problem} is the initial state in \eqref{eq:value-function}.  Strict convexity and radial symmetry show that
$\inf_{c\in\R^2}u_q(1,c)=u_q(1,0)$.  Therefore
\begin{equation*}
    \kappa_q
    =u_q(1,0)
    =\phi_q(0)
    =\kappa_q^{\mathrm{ode}},
\end{equation*}
which proves $\kappa_q
    =u_q(1,0)
    =\phi_q(0)$.
\end{proof}

\begin{proof}[Proof of \cref{lem:kappaqestimates}]
Taking zero control and $c=0$ in \eqref{eq:control-problem} gives the upper bound.  If $\sU=\sG+\sV$ is admissible, then
\begin{equation*}
 E[\sU\cdot\sG]=E\bra{|\sG|^2}=1
\end{equation*}
because $\sG\in\Ran\Qscale$ and
$\sV\in\Ran\Pscale$.  H\"older's inequality and the explicit formula for moments of $\sG$ give
\begin{equation*}
 1\le(E\bra{|\sU|^q})^{1/q}
 \Gamma\left(1+\frac p2\right)^{1/p},
\end{equation*}
which is the lower bound. For the strict inequality, let
\begin{equation*}
 \sG_s:=\int_0^s\int_0^\pi e_\theta\,\sw^1(dr,d\theta)
\end{equation*}
be the zero-control state. Set  $H(r,x)=D_x^2u_q(r,x)$ and 
\begin{equation*}
    \cD_q(r,x)
 :=\frac14\trace H(r,x)-\frac12\sqrt{\det H(r,x)}
 =\frac14\left(\sqrt{\lambda_{\mathrm{rad}}(r,x)}
 -\sqrt{\lambda_{\mathrm{tan}}(r,x)}\right)^2\ge0,
\end{equation*}
where $\lambda_{\mathrm{rad}}$ (respectively $\lambda_{\mathrm{tan}}$) are the radial (respectively the tangential) eigenvalues of $H$. Since 
\[
\frac12\trace H =\int_0^\pi
        e_\theta^\top
        He_\theta\frac{d\theta}{\pi},
\]
It\^o's formula up to time
$1-\varepsilon$ and \eqref{drift} give
\begin{equation*}
 E \bra{u_q(\varepsilon,\sG_{1-\varepsilon})}-u_q(1,0)
 =\int_0^{1-\varepsilon} \,
 E\bra{\cD_q(1-s,\sG_s)}\,ds.
\end{equation*}
The left-hand side converges to $E\bra{|\sG|^q}-u_q(1,0)$ by
\cref{lem:bellman-terminal} and
$\sG_{1-\varepsilon}\to\sG$ in $L^q$.  The right-hand side
converges by monotone convergence because its integrand is nonnegative.
Therefore
\begin{equation}\label{eq:strict-drift-identity}
 E\bra{|\sG|^q}-u_q(1,0)
 =\int_0^1E\bra{\,\cD_q(1-s,\sG_s)}\,ds.
\end{equation}
Assume $q\ne2$.  By the strict range statement in
\eqref{eq:h-range} and \eqref{eigenD2u}, the radial and tangential eigenvalues are unequal
whenever $r>0$ and $x\ne0$.  Hence $\cD_q(r,x)>0$ on that
set.  More quantitatively, on the compact set
\begin{equation*}
 \set{(r,x):\frac12\le r\le\frac34,
 1\le|x|\le2},
\end{equation*}
continuity gives $\cD_q\ge c_q>0$.  For
$s\in[1/4,1/2]$, the Gaussian vector $\sG_s$ is non-degenerate, and
\begin{equation*}
 \inf_{s\in[1/4,1/2]}
 P(1\le|\sG_s|\le2)>0.
\end{equation*}
The right-hand side of \eqref{eq:strict-drift-identity} is therefore strictly positive. Since $E\bra{|\sG|^q}=\Gamma(1+q/2)$, this concludes.
\end{proof}

%-------------------------------------------
\subsection{Projection Operators}

\begin{proof}[Proof of \cref{lem:scale-projections}]
The $L^2$ assertions follow from It\^o's isometry,
$P_\theta+Q_\theta=\id$, and $P_\theta Q_\theta=0$.
To prove the $L^r$ estimate, it is enough to consider
$\sF\in L_1^2\cap L_1^r$, which is dense in $L_1^r$.
Define
\begin{equation*}
 \sM_t:=E[\sF\mid\F_t],
 \qquad
 \widetilde{\sM}_t:=\sM_t-E[\sF],
\end{equation*}
and use \eqref{eq:white-noise-representation} to write
\begin{equation*}
 \widetilde{\sM}_t
 =\sum_{a=1}^2\int_0^t\int_0^\pi
 \Phi_{\sF}^a(s,\theta)\,\sw^a(ds,d\theta).
\end{equation*}
The centered martingale with terminal value $\Pscale\sF-E[\sF]$ is
\begin{equation*}
 \sN_t
 :=\sum_{a=1}^2\int_0^t\int_0^\pi
 P_\theta\Phi_{\sF}^a(s,\theta)\,\sw^a(ds,d\theta).
\end{equation*}
For an $\R^2$-valued continuous martingale $\sK$, write
$[\sK]^{\mathrm{tr}}:=[\sK^1]+[\sK^2]$. Since $P_\theta$ is an
orthogonal projection,
\begin{equation}\label{eq:scale-bracket-domination}
 [\sN]_t^{\mathrm{tr}}
 =\sum_{a=1}^2\int_0^t\int_0^\pi
 |P_\theta\Phi_{\sF}^a|^2\,\frac{d\theta}{\pi}\,ds
 \le [\widetilde{\sM}]_t^{\mathrm{tr}}.
\end{equation}
By \eqref{eq:scale-bracket-domination} and the vector-valued
Burkholder--Davis--Gundy inequalities,
\begin{equation*}
 \|\sN_1\|_{L^r}
 \le C_r\|\big([\widetilde{\sM}]_1^{\mathrm{tr}}\big)^{1/2}\|_{L^r}.
\end{equation*}
The reverse Burkholder--Davis--Gundy inequality and Doob's maximal inequality
then give
\begin{equation*}
 \|\big([\widetilde{\sM}]_1^{\mathrm{tr}}\big)^{1/2}\|_{L^r}
 \le C_r\|\sup_{t\le1}|\widetilde{\sM}_t|\|_{L^r}
 \le C_r\|\sF-E[\sF]\|_{L^r}
 \le C_r\|\sF\|_{L^r}.
\end{equation*}
Together with $|E[\sF]|\le\|\sF\|_{L^r}$, this proves
\begin{equation*}
 \|\Pscale\sF\|_{L^r}\le C_r\|\sF\|_{L^r}.
\end{equation*}
Replacing $P_\theta$ by $Q_\theta$ gives the same estimate for
$\Qscale$.  Both operators therefore extend uniquely to $L^r$, and the projection and complementarity identities pass to the extensions by density.

For square-integrable variables, \eqref{eq:scale-duality} follows from It\^o's isometry and the symmetry of $P_\theta$.  Approximation by $L_1^2\cap L_1^q$ and $L_1^2\cap L_1^{p}$, together with the bounds just proved, extends the identity to the stated exponents.  Finally, the annihilator of the range of a bounded projection is the kernel of its Banach adjoint.  Equation \eqref{eq:scale-duality} identifies that adjoint with $\Pscale$ on $L_1^{p}$, and complementarity gives \eqref{eq:annihilator-scale}.
\end{proof}

%-------------------------------------------
\subsection{Counting Lattice Points}

\begin{proof}[Proof of \cref{lem:log-polar-equidistribution}]
For an angular interval $I\subset[0,\pi]$, let
\begin{equation*}
    N_I(t):=\#\set{n\in\Nsc^+:0<|n|\le t,
    \ \theta_n\in I}.
\end{equation*}
% and let
% \begin{equation*}
%     S_I(t):=\set{r e_\theta:0<r\le t,\ \theta\in I}
% \end{equation*}
% be the corresponding planar sector. 

The planar Lipschitz principle \cite[Theorem~2.4]{widmer2012}, applied to the corresponding sector,
gives 
\begin{equation}\label{eq:sector-count}
    N_I(t)=\frac{|I|}{2}t^2+O_I(t+1).
\end{equation}

Interpreting the sum as a Stieltjes integral and using integration by parts yields, for fixed $0 \le a  < b \le 1$ and $R_\Lambda=\Lambda/(2\pi)$,
\begin{equation}\label{eq:sector-sum}
    \sum_{\substack{
        R_\Lambda^a<|n|\le R_\Lambda^b\\
        \theta_n\in I}}
        \frac1{|n|^2}
    =
    |I|(b-a)\log R_\Lambda+O_I(1).
\end{equation}

In particular 
\begin{equation}\label{eq:full-logarithmic-sum}
    \sum_{n\in \Nsc_\Lambda^+} \frac1{|n|^2}
    =
    \pi\log R_\Lambda+O(1).
\end{equation}

Endpoint conventions affect these sums only by $O_I(1)$. Hence,
\begin{equation*}
    \mu_\Lambda((a, b] \times I)
    \longrightarrow
    \frac{|I|}{\pi}(b-a),
\end{equation*}

which provides
\begin{equation*}
\mu_\Lambda  \rightharpoonup ds\,\frac{d\theta}{\pi}.
\end{equation*}

Every fixed tensor power converges accordingly. If $U_1, \ldots, U_N$ are independent uniform variables on $[0, 1]$, then Portmanteau's theorem and the union bound give
\begin{align*}
    \limsup_{\Lambda\to\infty}
    \mu_\Lambda^{\otimes N}
    \pa{D_{N,\delta}\times[0,\pi]^N}
    &\le
      P\pa{U_{(N)}-U_{(N-1)}\le\delta}\\
    &\le
      \sum_{i<j}P(|U_i-U_j|\le\delta)
      \le N(N-1)\delta.
\end{align*}

\end{proof}
\subsection*{Acknowledgements}  This work was supported by a public grant from the Fondation Mathématique Jacques Hadamard.

OpenAI GPT5.6 was used for general proof-checking, it suggested an improvement for the argument of \cref{lem:qone-sharp-truncation} based on that of \cite{Muller99} over the original one based on \cite{AST19}, and it performed the ODE study of \cref{ssec:ode}.
All mathematical arguments and conclusions in the paper were independently checked and verified by the authors, who take full responsibility for the content of the manuscript.

%%%%%%%%%%%%%%%%%%%%%%%%%%%%%%%%%%%%%%%%%%%%%%%%%%%%%%%%%%%%%%%%%%

\bibliography{biblio.bib}{}
\bibliographystyle{plain}

\end{document}